\documentclass[11pt,a4paper,notitlepage,oneside]{amsart}
 \pdfoutput=1 
\usepackage{etoolbox}
\usepackage{microtype}
\usepackage{amsmath}
\usepackage{amsthm}
\usepackage{braket}
\usepackage{pdfpages}
\usepackage[psamsfonts]{amssymb}
\usepackage[T1]{fontenc}
\usepackage[english]{babel}
\usepackage[utf8]{inputenc}
\usepackage{enumitem}
\usepackage{mathtools}
\usepackage{caption,subfig}
\usepackage{rotating}
\usepackage[makeroom]{cancel}
\usepackage[normalem]{ulem}
\usepackage{color,soul}
\usepackage{scalerel}
\usepackage{relsize}

\usepackage[autostyle,italian=guillemets]{csquotes}
\usepackage[style=alphabetic,firstinits=true,maxbibnames=99,backend=bibtex]{biblatex}
\bibliography{bibliography}
\allowdisplaybreaks 

\usepackage{pgf,tikz,pgfplots}

\usepackage{mathrsfs}
\usetikzlibrary{arrows}

\usepackage{fancyhdr}
\usepackage{tikz-cd}
\usepackage{amscd}
\usepackage[a4paper, bottom=3cm, left=3cm,right=3cm]{geometry}
\usepackage{hyperref}
\hypersetup{hidelinks}

\usepackage{subfiles}
\usepackage{my_shortcuts}

\author{Riccardo Zuffetti}
\title{Unfolding and injectivity of the Kudla--Millson lift of genus~1}
\address{
Fachbereich Mathematik, Technische Universit\"at Darmstadt, Schlossgartenstrasse 7,
D–64289 Darmstadt, Germany\\
}
\email{zuffetti@mathematik.tu-darmstadt.de, riccardo.zuffetti@gmail.com}

\newcommand{\myblack}{black}

\begin{document}
\maketitle
	\begin{abstract}
	We unfold the theta integrals defining the Kudla--Millson lift of genus~$1$ associated to even lattices of signature~$(b,2)$, where~$b>2$.
	This enables us to compute the Fourier expansion of such defining integrals and prove the injectivity of the Kudla--Millson lift.
	Although the latter result has been already proved in~\cite{br;converse}, our new procedure has the advantage of paving the ground for a strategy to prove the injectivity of the lift also for the cases of general signature and of genus greater than~$1$.
	\end{abstract}
	\tableofcontents
	
	\section{Introduction}

	We consider the Kudla--Millson lift as a linear map from a space of elliptic cusp forms to the space of closed~$2$-forms on some orthogonal Shimura varieties.
	Starting from the foundational work of Kudla an Millson~\cite{kumi;harmI} \cite{kumi;harmII} \cite{kumi;intnum}, such a lift has attracted much interest.
	In fact, it is strictly related with Borcherds' singular theta lift~\cite{brfutwo}, and its injectivity provides a way to study the geometry of orthogonal Shimura varieties by means of modular forms~\cite{br;borchp}~\cite{br;converse}~\cite{brmo}~\cite{zuf;conesspeccy}.
	
	In this article we apply Borcherds' formalism~\cite{bo;grass} to unfold the theta integrals that arise from the very definition of the Kudla--Millson lift.
	As an application, we compute the Fourier expansion of such integrals and prove that the Kudla--Millson lift is injective in many cases.
	The latter result is the same as~\cite[Theorem~$5.3$]{br;converse}, but it is here proved in a different way.
	
	The new procedure illustrated in this paper has the advantage of paving the ground for a strategy to prove the injectivity of the Kudla--Millson lift in the case of general signature, generalizing~\cite{brfu}, as well as the cases of genus higher than~$1$.
	\textcolor{\myblack}{
	This has been recently implemented by the author and Metzler~\cite{metzlerzuffetti}, and the author and Kiefer~\cite{kieferzuffetti} respectively.
	}
	
	\textcolor{\myblack}{
	The theta kernel of the Kudla--Millson lift may be constructed as a theta form associated to the reductive dual pair~$(\SL_2(\RR),\bigO(n,2))$.
	The theta correspondence in this setting has been of great interest in several other articles, for instance in the works of Oda~\cite{oda} and Rallis--Schiffmann~\cite{rallis-schiffmann},~\cite{rallis}, where theta lifts from elliptic modular forms (with level) to orthogonal modular forms were studied.
	}
	
	We now explain the results of this article in more details.	
	Let~$\big(L,(\cdot{,}\cdot)\big)$ be a non-degenerate even lattice of signature~$(b,2)$, where~$b>2$.
	We denote by~$q(\lambda)=(\lambda,\lambda)/2$ the quadratic form of~$L$, and write~$\lambda^2$ in place of~$(\lambda,\lambda)$, for every~$\lambda\in L$.
	To simplify the illustration, we assume~$L$ to be \emph{unimodular}.
	Let~$k=1+b/2$.
	This is an even integer, as one can easily deduce from the well-known classification of unimodular lattices.
	
	Let~$V=L\otimes\RR$.
	The Hermitian symmetric domain~$\hermdom$ associated to the linear algebraic group~${G=\GG(V)}$ may be realized as the Grassmannian~$\Gr(L)$ of negative definite planes in~$V$.
	Let~${X_\Gamma=\Gamma\backslash\hermdom}$ be the orthogonal Shimura variety arising from a subgroup~$\Gamma$ of finite index in~$\GG(L)$.
	
	Kudla and Millson constructed a $G$-invariant Schwartz function $\varphi_{\text{KM}}$ on $V$ with values in the space $\mathcal{Z}^2(\hermdom)$ of closed differential $2$-forms on $\hermdom$.
	An explicit formula for~$\varphi_{\text{KM}}$ is provided in Section~\ref{sec;compDp2}.
	We denote by~$\omega_\infty$ the Schrödinger model of the Weil representation of~$\SL_2(\RR)$ acting on the space~$\mathcal{S}(V)$ of Schwartz functions on~$V$; see Definition~\ref{defi;schrmod} for details.
	\begin{defi}\label{def;KMthetaformgenus1}
	The Kudla--Millson theta form associated to~$L$ is defined as
	\bes
	\Theta(\tau,z,\varphi_{\text{KM}})=y^{-k/2}\sum_{\lambda\in L}\big(\omega_\infty(g_\tau)\varphi_{\text{KM}}\big)(\lambda,z),
	\ees
	for every $\tau=x+iy\in\HH$ and $z\in\Gr(L)$, where $g_\tau=\left(\begin{smallmatrix}
	1 & x\\ 0 & 1
	\end{smallmatrix}\right)\big(\begin{smallmatrix}
	\sqrt{y} & 0\\ 0 & \sqrt{y}^{-1}
	\end{smallmatrix}\big)$ is the standard element of~$\SL_2(\RR)$ mapping~${i\in\HH}$ to~$\tau$.
	\end{defi}
	In the variable~$\tau$, the Kudla--Millson theta form transforms like a (non-holomorphic) modular form of weight~$k$ with respect to~$\SL_2(\ZZ)$.	
	In the variable~$z$, it defines a closed~$2$-form on~$\hermdom$ and~$X_\Gamma$.
	
	Let~$S^k_1$ be the space of weight~$k$ elliptic cusp forms with respect to the full modular group~$\SL_2(\ZZ)$.
	\begin{defi}\label{def;KMliftgen1}
	The \emph{Kudla--Millson lift} of genus $1$ associated to~$L$ is the map
	\be\label{eq;liftingLambda}
	\KMliftbase\colon S^k_1\longrightarrow\mathcal{Z}^2(X_\Gamma), \qquad f\longmapsto\KMliftbase(f)=\int_{\SL_2(\ZZ)\backslash\HH}y^k f(\tau)\overline{\Theta(\tau,z,\varphi_{\text{\rm KM}})}\,\frac{dx\, dy}{y^2},
	\ee
	where $\frac{dx\, dy}{y^2}$ is the standard~$\SL_2(\ZZ)$-invariant volume element of~$\HH$.
	\end{defi}
	
	In Section~\ref{sec;KMthetagen1} we rewrite~$\Theta(\tau,z,\varphi_{\text{KM}})$ in terms of Siegel theta functions~$\Theta_L$ attached to certain homogeneous polynomials~$\polab$ of degree~$(2,0)$ defined on the standard quadratic space~$\RR^{b,2}$; see~\eqref{eq;comppolPab} for the definition of such polynomials.
	The Siegel theta functions~$\Theta_L$ were introduced by Borcherds in~\cite{bo;grass}.
	
	As explained in Section~\ref{sec;theunfofKMgen1}, is it possible to rewrite the lift~$\KMliftbase(f)$ as
	\be\label{eq;introKMl1nonunfold}
	\KMliftbase(f)=\sum_{\alpha,\beta=1}^b \Big(\underbrace{\int_{\SL_2(\ZZ)\backslash\HH} y^{k+1} f(\tau) \overline{\Theta_L(\tau,g,\polab)}\,\frac{dx\,dy}{y^2}}_{\eqqcolon\intfunct(g)}\Big)
	\cdot
	g^*\big(\omega_{\alpha,b+1}\wedge\omega_{\beta,b+2}\big),
	\ee
	where~$g\in G$ is any isometry mapping~$z$ to a fixed base point~$z_0$ of~$\mathcal{D}$, and~${g^*\big(\omega_{\alpha,b+1}\wedge\omega_{\beta,b+2}\big)}$ is a cotangent vector of~$\bigwedge^2 T_z^* \mathcal{D}$; see Definition~\ref{defi;KMfunct} for details on the construction of the latter.
	We refer to the integral functions~$\intfunct\colon G\to\CC$ appearing in~\eqref{eq;introKMl1nonunfold} as the \emph{defining integrals} of the Kudla--Millson lift of~$f$.
	
	The idea of this paper is to apply the formalism of Borcherds~\cite{bo;grass} to unfold the defining integrals~$\intfunct$, rewriting them over the simpler unfolded domain~$\Gamma_\infty\backslash\HH$, where~$\Gamma_\infty$ is the subgroup of translations in~$\SL_2(\ZZ)$.
	\textcolor{\myblack}{Such an unfolding, which depends on the choice of}
	a splitting~$L=\brK\oplus U$ for some Lorentzian sublattice~$\brK$ and some hyperbolic plane~$U$, \textcolor{\myblack}
	{
	is carried over in Section~\ref{sec;theunfofKMgen1}; see Corollary~\ref{cor;unfolded!} for a precise statement.
	}
		
	If a complex valued function defined over~$G$ is invariant with respect to (Eichler transformations arising from) some Lorentzian sublattice of~$L$, then it admits a Fourier expansion.
	Although this general principle is classical in the literature, for the sake of completeness we provide an overview of it in Section~\ref{sec;FexpLlorinvfuncts}.
	
	\textcolor{\myblack}{In Section~\ref{sec;compFexpafterunfold} we use the unfolding of~$\intfunct$ to compute the Fourier expansion of the defining integrals~$\intfunct$}, obtaining the following result.
	We denote by~$c_n(f)$ the~$n$-th Fourier coefficient of the cusp form~$f$.
	\textcolor{\myblack}{
	Let~$\genU,\genUU$ be the standard generators of the hyperbolic plane~$U$.
	For every~${g\in G}$, we denote by~$z\in\Gr(L)$ the plane mapping to the base point~$z_0$ under~$g$, and denote by~$v_z$ the orthogonal projection of any~$v\in L\otimes\RR$ to~$z$.
	We also denote by~$w^\perp$ the orthogonal complement of~$u_{z^\perp}$ in~$z^\perp$.
	The linear map~$\borw \colon L\otimes\RR\to L\otimes\RR$, $\borw(\genvec)=g(\genvec_{w^\perp}+\genvec_w)$, induces a split of the polynomial~$\polab$ in terms of polynomials~$\polw{(\alpha,\beta),\borw}{h^+}{0}$ of degree~$(2-h^+,0)$; see~\eqref{eq;bordefimplpol} for details.}
	\begin{thm}\label{thm;Fexpintro}
	Let~$f\in S^k_1$, and let~$\intfunct\colon G\to\CC$ be the defining integrals of the Kudla--Millson lift~$\KMliftbase(f)$.	
	The Fourier coefficient of~$\intfunct$ associated to~$\lambda\in\brK$, such that~$q(\lambda)>0$, is
	\ba\label{eq;foucoefcorintro}
	&\frac{\sqrt{2}}{|\genU_{z^\perp}|}
	\sum_{h^+=0}^2
	\sum_{\substack{t\ge 1\\ t|\lambda}}
	\Big(\frac{t}{2i}\Big)^{h^+}
	c_{q(\lambda)/t^2}(f)
	\int_0^{+\infty} y^{k-h^+-3/2}\exp\Big(-\frac{2\pi y \lambda_{w^\perp}^2}{t^2} -
	\frac{\pi t^2}{2y\genU_{z^\perp}^2}
	\Big)
	\\
	&\quad\times\exp(-\Delta/8\pi y)\big(\polw{(\alpha,\beta),\borw}{h^+}{0}\big)\big(\borw(\lambda/t)\big)dy,
	\ea
	where we say that an integer $t\ge1$ divides $\lambda\in\brK$, in short~$t|\lambda$, if and only if~$\lambda/t$ is still a lattice vector in $\brK$.
	
	The Fourier coefficient of~$\intfunct$ associated to~$\lambda=0$, i.e.\ the constant term of the Fourier series, is
	\ba\label{eq;intconstcoefintro}
	\int_{\SL_2(\ZZ)\backslash\HH}\frac{y^{k+1/2}f(\tau)}{\sqrt{2}\, |\genU_{z^\perp}|}
	\,
	\overline{\Theta_{\brK}(\tau,\borw,\polw{(\alpha,\beta),\borw}{0}{0})}\,\frac{dx\,dy}{y^2}.
	\ea
	
	In all remaining cases, the Fourier coefficients are trivial.
	\end{thm}
	In Section~\ref{sec;injKMgenus1} we illustrate how to deduce the injectivity of~$\KMliftbase$ by means of the Fourier expansion provided by Theorem~\ref{thm;Fexpintro}.
	The idea is as follows.
	\textcolor{\myblack}{Since the cotangent vectors~$\omega_{\alpha,b+1}\wedge\omega_{\beta,b+2}$ are linearly independent, if the lift~$\KMliftbase(f)$ of a cusp form~$f$ vanishes, then all defining integrals~$\intfunct$ in~\eqref{eq;introKMl1nonunfold} are zero.
	This implies that all Fourier coefficients of~$\intfunct$ appearing in Theorem~\ref{thm;Fexpintro} are zero for every~$g\in G$.
	We show that for every~$\lambda\in\brK$ of positive norm there exists an isometry~$g\in G$ such that some of the integrals appearing in~\eqref{eq;foucoefcorintro} are non-zero; see Lemma~\ref{lemma;technnonzeropolinprinj}.
	An easy inductive argument on the primitivity of~$\lambda$ implies that the Fourier coefficient~\eqref{eq;foucoefcorintro} vanishes for such choices of~$\lambda$ and~$g$ only if~$c_\lambda(f)=0$.
	Summarizing, if~$\KMliftbase(f)=0$, then all Fourier coefficients of~$f$ vanish, therefore~$f$ is trivial.}
	
	The results above are illustrated in the case of even \emph{unimodular} lattices~$L$ of signature~$(b,2)$, where~$b>2$.
	In Section~\ref{sec;nonunimodlattgen1} we generalize them to non-unimodular lattices.
	We also provide an alternative proof of the known injectivity of~$\KMliftbase$ in this setting, namely~\cite[Theorem~$5.3$]{br;converse}, stated below.
	For every integer~$N$, we write~$U(N)$ for the lattice~$U$ as a~$\ZZ$-module but equipped with the rescaled quadratic form~$Nq(\cdot)$.
	\begin{thm}[Bruinier]
	Let $L$ be an even lattice of signature $(b,2)$, where~$b>2$, that splits off~${U(N)\oplus U}$, for some positive integer~$N$.
	The Kudla--Millson theta lift~$\KMliftbase$ is injective. 
	\end{thm}
	
	\subsection*{Acknowledgments}
	We are grateful to Jan Bruinier and Martin Möller for suggesting this project, for their encouragement, and for several useful conversations we shared.
	\textcolor{\myblack}{We would also like to thank Xiaoyu Zhang for useful discussions, as well as the anonymous referee for a careful reading of the present paper.}
	This work is a result of our PhD~\cite{zufthesis}, which was founded by the LOEWE research unit ``Uniformized Structures in Arithmetic and Geometry'', and by the Collaborative Research Centre TRR~326 ``Geometry and Arithmetic of Uniformized Structures'', project number~444845124.
	
	\section{The Kudla--Millson Schwartz function}\label{sec;compDp2}

	
	Let~$V$ be a real vector space endowed with a symmetric bilinear form~$(\cdot{,}\cdot)$ of signature~$(b,2)$, where~$b>2$.
	Its associated quadratic form is defined as $q(\cdot)=(\cdot{,}\cdot)/2$.
	In this section we provide an explicit formula of the Kudla--Millson Schwartz function~$\varphi_{\text{\rm KM}}$ attached to~$V$, following the wording of~\cite[Section~$2$ and Section~$4$]{brfutwo} and~\cite[Section~$7$]{ku;algcycle}.\\
	
	Let $(\basevec_j)_j$ be an orthogonal basis of $V$ such that $(\basevec_\alpha,\basevec_\alpha)=1$ for every $\alpha=1,\dots,b$, and~${(\basevec_\mu,\basevec_\mu)=-1}$ for $\mu=b+1,b+2$.
	We denote the corresponding coordinate functions by~$x_\alpha$ and~$x_\mu$.
	The choice of~$(\basevec_j)_j$ induces an isometry~${g_0\colon V\to\RR^{b,2}}$, where~$\RR^{b,2}$ is the real space~$\RR^{b+2}$ endowed with the standard quadratic form of signature~$(b,2)$ defined as
	\bes
	q_0\big((x_1,\dots,x_{b+2})^t\big)=\sum_{j=1}^b x_j^2/2 - x_{b+1}^2/2 - x_{b+2}^2/2,\qquad\text{for every $(x_1,\dots,x_{b+2})^t\in\RR^{b+2}$.}
	\ees
	
	The Grassmannian associated to $V$ is the set of negative definite planes in~$V$, namely
	\bes
	\Gr(V)=\{z\subset V : \text{$\dim z=2$ and $(\cdot{,}\cdot)|_z<0$}\}.
	\ees
	The plane~$z_0$ spanned by~$\basevec_{b+1}$ and~$\basevec_{b+2}$ is the \emph{base point} of~$\Gr(V)$.
	The Hermitian symmetric space arising as the quotient~$\hermdom=G/K$, where~$G=\SO(V)\cong\SO(b,2)$ and~$K$ is the maximal compact subgroup of~$G$ stabilizing $z_0$, may be identified with~$\Gr(V)$; see~\cite[Part~$2$, Section~$2.4$]{1-2-3}.
	From now on, we write~$\hermdom$ and~$\Gr(V)$ interchangeably.
	
	For every $z\in \hermdom$, we define the \emph{standard majorant} $(\cdot{,}\cdot)_z$ as
	\be\label{eq;standardmajorant}
	(\genvec,\genvec)_z=(\genvec_{z^\bot},\genvec_{z^\bot})-(\genvec_z,\genvec_z),
	\ee
	where $\genvec=\genvec_z+\genvec_{z^\perp}\in V$ is written with respect to the decomposition~${V=z\oplus z^\perp}$.
	We will often write~$v^2$ in place of~$(v,v)$, and the standard majorant as~$(\genvec,\genvec)_z=\genvec_{z^\perp}^2-\genvec_z^2$.
	
	Let $\mathfrak{g}$ be the Lie algebra of $G$, and let $\mathfrak{g}=\mathfrak{p}+\mathfrak{k}$ be its Cartan decomposition.
	It is well-known that~$\mathfrak{p}\cong\mathfrak{g}/\mathfrak{k}$ is isomorphic to the tangent space of~$\hermdom$ at the base point~$z_0$.
	With respect to the basis of $V$ chosen above, we have
	\be\label{eq;mathfracpiso}
	\mathfrak{p}\cong\left\{\left(
	\begin{smallmatrix}
	0 & X\\
	X^t & 0
	\end{smallmatrix}
	\right)\,|\,\text{$X\in\Mat_{b,2}(\RR)$}\right\}\cong \Mat_{b,2}(\RR).
	\ee
	We may assume that the chosen isomorphism is such that the complex structure on~$\mathfrak{p}$ is given as the right-multiplication by~$J=\big(\begin{smallmatrix}0 & 1\\ -1 & 0\end{smallmatrix}\big)\in\GL_2(\RR)$ on~$\Mat_{b,2}(\RR)$.
	
	To simplify the notation, we write~$e(t)=\exp(2\pi i t)$, for every~$t\in\CC$, and denote by~${\sqrt{t}=t^{1/2}}$ the principal branch of the square root, so that~$\arg(\sqrt{t})\in(-\pi/2,\pi/2]$.
	If~$s\in\CC$, we define~$t^s=e^{s\Log(t)}$, where~$\Log(t)$ is the principal branch of the logarithm.
	\begin{defi}\label{defi;schrmod}
	We denote by $\omega_\infty$ the Schrödinger model of (the restriction of) the Weil representation of~$\Mp_2(\RR)\times \bigO(V)$ acting on the space~$\mathcal{S}(V)$ of Schwartz functions on~$V$.
	The action of~$\bigO(V)$ is defined as
	\bes
	\omega_\infty(g)\varphi(\genvec)=\varphi\big(g^{-1}(\genvec)\big),
	\ees
	for every~$\varphi\in\mathcal{S}(V)$ and~$g\in \bigO(V)$.
	The action of~$\Mp_2(\RR)$ is given by
	\ba\label{eq;actMp2Schrod}
	\omega_\infty\left(\begin{smallmatrix}
	1 & x\\ 0 & 1
	\end{smallmatrix}\right)\varphi(\genvec)&=e(x q(\genvec))\varphi(\genvec),\quad\text{for every $x\in\RR$},\\
	\omega_\infty\left(\begin{smallmatrix}
	a & 0\\ 0 & a^{-1}
	\end{smallmatrix}\right)\varphi(\genvec)&=a^{(b+2)/2}\varphi(a\genvec),\quad\text{for every $a>0$},\\
	\omega_\infty(S)\varphi(\genvec)&=\sqrt{i}^{2-b}\widehat{\varphi}(-\genvec),
	\ea
	where $S=\big(\big(\begin{smallmatrix}
	0 & -1\\ 1 & 0
	\end{smallmatrix}\big),\sqrt{\tau}\big)$, and $\widehat{\varphi}(\xi)=\int_V\varphi(\genvec)e^{2\pi i (\genvec,\xi)}dv$ is the Fourier transform of $\varphi$.
	\end{defi}
	
	The \emph{standard Gaussian of}~$\RR^{b,2}$ is defined as
	\bes
	\varphi_0(x_1,\dots,x_{b+2})=e^{-\pi\sum_{j=1}^{b+2}x_j^2},\qquad\text{for every $(x_1,\dots,x_{b+2})^t\in\RR^{b+2}$}.
	\ees
	The \emph{standard Gaussian of} $V$ is the Schwartz function~$\varphi_0\circ g_0$, where~$g_0$ is the isometry arising from the choice of the basis $(\basevec_j)_j$ of $V$.
%

	Let~\textcolor{\myblack}{$\mathcal{A}^2(\hermdom)$, resp.}~$\mathcal{Z}^2(\hermdom)$, be the space of \textcolor{\myblack}{differential, resp.}~closed,~$2$-forms on~$\hermdom$.
	The group~$G$ acts on~$\textcolor{\myblack}{\mathcal{A}^2}(\hermdom)$ by pullback.
	We say that $\varphi\in\mathcal{S}(V)\otimes \textcolor{\myblack}{\mathcal{A}^2}(\hermdom)$ is~$G$-invariant if
	\be\label{eq;clarinf2forms}
	g^*\varphi\big(g(\genvec)\big)=\varphi(\genvec)\qquad\text{for every $\genvec\in V$},
	\ee
	where $g^*\varphi(\genvec)$ is the pullback of $\varphi(\genvec)\in \textcolor{\myblack}{\mathcal{A}^2}(\hermdom)$ induced by the action of $g$ on $\hermdom$.
	
	We remark that
	\be\label{eq;isoschwartzbigwedge}
	\big[\mathcal{S}(V)\otimes \textcolor{\myblack}{\mathcal{A}^2}(\hermdom)\big]^G\cong\Big[\mathcal{S}(V)\otimes{\bigwedge}^2(\mathfrak{p}^*)\Big]^K,
	\ee
	where the isomorphism is given by evaluating~$G$-invariant functions on the left-hand side of~\eqref{eq;isoschwartzbigwedge} at the base point~$z_0$ of~$\hermdom$.
	Therefore, we can define any $G$-invariant function~$\varphi$ in~$\big[\mathcal{S}(V)\otimes \textcolor{\myblack}{\mathcal{A}^2}(\hermdom)\big]^G$ firstly as an element of~$\big[\mathcal{S}(V)\otimes{\bigwedge}^2(\mathfrak{p}^*)\big]^K$, and then spread it to the whole~$\hermdom$ under the action of~$G$.
	The Kudla--Millson Schwartz function~$\varphi_{\text{\rm KM}}$ is constructed in this way, as we now quickly recall.
	
	\begin{defi}\label{defi;KMfunct}
	We denote by $X_{\alpha,\mu}$, where~$1\le\alpha\le b$ and~$b+1\le\mu\le b+2$, the basis vectors of~$\Mat_{b,2}(\RR)$ given by matrices with~$1$ at the~$(\alpha,\mu-b)$-th entry and zero otherwise.
	These vectors provide a basis of~$\mathfrak{p}$ under the isomorphism~\eqref{eq;mathfracpiso}.
	Let~$\omega_{\alpha,\mu}$ be the element of the corresponding dual basis which extracts the $(\alpha,\mu-b)$-th coordinate of elements in~$\mathfrak{p}$, and let~$A_{\alpha,\mu}$ be the left multiplication by~$\omega_{\alpha,\mu}$.
	The function~$\varphi_{\text{\rm KM}}$ is defined applying the operator
	\bes
	\KMoper=\frac{1}{2}\prod_{\mu=b+1}^{b+2}\Big[\sum_{\alpha=1}^b\Big(x_\alpha-\frac{1}{2\pi}\frac{\partial}{\partial x_\alpha}\Big)\otimes A_{\alpha,\mu}\Big]
	\ees
	to the standard Gaussian $(\varphi_0\circ g_0)\otimes 1\in \big[\mathcal{S}(V)\otimes{\bigwedge}^0(\mathfrak{p}^*)\big]^K$, namely~$\varphi_{\text{\rm KM}}=\KMoper\big((\varphi_0\circ g_0)\otimes 1\big).$
	\end{defi}
	\textcolor{\myblack}{It is easy to rewrite~$\varphi_{\text{\rm KM}}\in\big[\mathcal{S}(V)\otimes{\bigwedge}^2(\mathfrak{p}^*)\big]^K$ explicitly as}
	\ba\label{eq;finphiKMb2}
	\varphi_{\text{\rm KM}}(\genvec,z_0)
	=\sum_{\alpha,\beta=1}^b \Big(\Gpol_{(\alpha,\beta)}\varphi_0\Big)\big(g_0(\genvec)\big)\omega_{\alpha, b+1}\wedge \omega_{\beta, b+2},
	\ea
	where
	\be\label{eq;comppolPab}
	\Gpol_{(\alpha,\beta)}\big(g_0(v)\big)\coloneqq\begin{cases}
	\polab\big(g_0(\genvec)\big), & \text{if $\alpha\neq\beta$,}\\
	\polab\big(g_0(\genvec)\big)-\frac{1}{2\pi}, & \text{otherwise,}
	\end{cases}
	\qquad\text{and}\qquad
	\polab\big(g_0(\genvec)\big)\coloneqq 2x_\alpha x_\beta,
	\ee
	for every $\genvec\in V$ with~$g_0(\genvec)=(x_1,\dots,x_{b+2})^t\in\RR^{b,2}$.
	It is easy to check that
	\bes
	\Gpol_{(\alpha,\beta)}\big(g_0(\genvec)\big)=
	\begin{cases}
	\frac{1}{2}H_1(x_\alpha)H_1(x_\beta) & \text{if $\alpha\neq\beta$,}\\
	\frac{1}{4\pi}H_2(\sqrt{2\pi}x_\alpha) & \text{otherwise,}
	\end{cases}
	\ees
	where $H_n(t)$ is the $n$-th Hermite polynomial.
	
	\begin{rem}\label{rem;finformphiKMglobal}
	In~\eqref{eq;finphiKMb2} we constructed~$\varphi_{\text{\rm KM}}$ as a~$K$-invariant function in~${\mathcal{S}(V)\otimes{\bigwedge}^2(\mathfrak{p}^*)}$.
	To construct a global $G$-invariant function in~$\big[\mathcal{S}(V)\otimes \textcolor{\myblack}{\mathcal{A}^2}(\hermdom)\big]^G$, we may spread~$\varphi_{\text{\rm KM}}$ on the whole~$G$ by means of~\eqref{eq;isoschwartzbigwedge} as follows.
	Let~$z\in \hermdom$, and let~$g\in G$ be any isometry mapping~$z$ to~$z_0$.
	By virtue of~\eqref{eq;clarinf2forms} we have
	\be\label{eq;finformglobspreadphikm}
	\varphi_{\text{\rm KM}}(\genvec,z)=
	g^*\varphi_{\text{\rm KM}}\big(g(\genvec),z_0\big)
	=\sum_{\alpha,\beta=1}^b \Big(\Gpol_{(\alpha,\beta)}\varphi_0\Big)\big(g_0\circ g(\genvec)\big)\cdot g^*(\omega_{\alpha, b+1}\wedge \omega_{\beta, b+2}).
	\ee
	Since~$\varphi_{\text{\rm KM}}$ is~$K$-invariant at the base point~$z_0$, the value~$\varphi_{\text{\rm KM}}(v,z)$ given by~\eqref{eq;finformglobspreadphikm} does not depend on the choice of the isometry~$g$ mapping~$z$ to~$z_0$.
	\textcolor{\myblack}{Furthermore, the function~$\varphi_{\text{\rm KM}}$ is a \emph{closed} form on~$\hermdom$; see~\cite{kumi;harmI}.}
	\end{rem}
	
	
	\section{The Kudla--Millson theta form}\label{sec;KMthetagen1}
	
	After a quick overview of some well-known results, in this section we provide an explicit formula of~$\Theta(\tau,z,\varphi_{\text{\rm KM}})$ by means of the one of~$\varphi_{\text{\rm KM}}$ computed in Section~\ref{sec;compDp2}.
	We then introduce Borcherds' formalism~\cite{bo;grass}, and show how to rewrite the Kudla--Millson theta form in terms of Siegel theta functions.\\
	
	Let~$\big(L,(\cdot{,}\cdot)\big)$ be a unimodular lattice of signature $(b,2)$, where~$b>2$.
	We fix once and for all an integer $k=1+b/2$ and an orthogonal basis~$(\basevec_j)_j$ of $V=L\otimes\RR$ such that~${\basevec_j^2=1}$, for every $j=1,\dots,b$, and $\basevec_{b+1}^2=\basevec_{b+2}^2=-1$.
	The choice of such a basis induces an isometry~${g_0\colon V\to \RR^{b,2}}$.
	We denote the Grassmannian~$\Gr(V)$ also by~$\Gr(L)$.
	This is identified with the Hermitian symmetric domain~$\hermdom$ arising from~$G$.
	
	We recall from Definition~\ref{def;KMthetaformgenus1} that the Kudla--Millson theta form is defined as
	\bes
	\Theta(\tau,z,\varphi_{\text{KM}})=y^{-k/2}\sum_{\lambda\in L}\big(\omega_\infty(g_\tau)\varphi_{\text{KM}}\big)(\lambda,z),
	\ees
	for every $\tau=x+iy\in\HH$ and $z\in\Gr(L)$, where $g_\tau=\left(\begin{smallmatrix}
	1 & x\\ 0 & 1
	\end{smallmatrix}\right)\Big(\begin{smallmatrix}
	\sqrt{y} & 0\\ 0 & \sqrt{y}^{-1}
	\end{smallmatrix}\Big)$ is the standard element of~$\SL_2(\RR)$ mapping~${i\in\HH}$ to~$\tau$.
	
	Let $A^k_1$ be the space of analytic functions on $\HH$ satisfying the classical weight~$k$ modular transformation property with respect to~$\SL_2(\ZZ)$.
	The Kudla--Millson theta form~$\Theta(\tau,z,\varphi_{\text{\rm KM}})$ is a non-holomorphic modular form with respect to the variable $\tau\in\HH$, and a closed $2$-form with respect to the variable $z\in\hermdom$, in short~$\Theta(\tau,z,\varphi_{\text{\rm KM}})\in A^k_1\otimes\mathcal{Z}^2(\hermdom)$.
	In fact, the Kudla--Millson theta form is~$\Gamma$-invariant, for every subgroup~$\Gamma$ of finite index in~$\GG(L)$.
	This implies that~$\Theta(\tau,z,\varphi_{\text{\rm KM}})$ descends to an element of~$A^k_1\otimes\mathcal{Z}^2(X_\Gamma)$, where~${X_\Gamma=\Gamma\backslash\hermdom}$ is the orthogonal Shimura variety arising from~$\Gamma$.
	
	Kudla and Millson showed that the $n$-th Fourier coefficient of~$\Theta(\tau,z,\varphi_{\text{\rm KM}})$ is a Poincar\'e dual form for the~$n$-th Heegner divisor on~$X_\Gamma$.
	Moreover, they proved that the cohomology class~$[\Theta(\tau,z,\varphi_{\text{\rm KM}})]$ is a holomorphic modular form of weight~$k$ with values in~$H^{1,1}(X_\Gamma)$, and coincides with Kudla's generating series of Heegner divisors; see~\cite{kumi;intnum} and~\cite[Theorem~3.1]{Kudla;speccycl}.
		
	Using the spread~\eqref{eq;finformglobspreadphikm} of~$\varphi_{\text{\rm KM}}$, we may rewrite the Kudla--Millson theta form as
	\ba\label{eq;KMdeg1recallexpl}
	\Theta(\tau,z,\varphi_{\text{\rm KM}})=\sum_{\alpha,\beta=1}^b\underbrace{y^{-k/2}\sum_{\lambda\in L}\Big(\omega_\infty(g_\tau)(\Gpol_{(\alpha,\beta)}\varphi_0)\Big)\big(g_0\circ g(\lambda)\big)}_{\eqqcolon \FFab(\tau,g)}\cdot\, g^*(\omega_{\alpha,b+1}\wedge\omega_{\beta,b+2}),
	\ea
	where $g\in G$ is any isometry of $V=L\otimes\RR$ mapping $z$ to $z_0$, and~$\Gpol_{(\alpha,\beta)}$ is the polynomial defined in~\eqref{eq;comppolPab}.
	Since the Kudla--Millson Schwartz function $\varphi_{\text{\rm KM}}$ is the spread to the whole~$\hermdom=\Gr(L)$ of an element of~$\mathcal{S}(V)\otimes{\bigwedge}^2(\mathfrak{p}^*)$ which is~$K$-invariant, the definition of~$\Theta(\tau,z,\varphi_{\text{\rm KM}})$ does not depend on the choice of~$g$ mapping~$z$ to~$z_0$.
	
	One of the goals of Section~\ref{sec;follBorcherds} is to rewrite the auxiliary functions $\FFab(\tau,g)$ arising as in~\eqref{eq;KMdeg1recallexpl} in terms of Siegel theta functions.
		
	
	\subsection{Siegel theta functions}\label{sec;Siegtheta}
	Let~$M$ be an indefinite even unimodular lattice of general signature~$(p,q)$.
	In this section we quickly recall how to construct Siegel theta functions attached to~$M$, as in~\cite[Section~$4$]{bo;grass}.
	This is ancillary to Section~\ref{sec;follBorcherds}, where we will rewrite the Kudla--Millson theta form associated to a lattice~$L$ of signature~$(b,2)$ in terms of Siegel theta functions attached to~$L$ and certain Lorentzian sublattices.\\
	
	We denote by~$\Gr(M)$ the Grassmannian of negative-definite subspaces of dimension~$q$ in~$M\otimes\RR$.
	This is (a realization of) the symmetric domain associated to the linear algebraic group~$G=\SO(M\otimes\RR)$.
	We fix once for all an isometry~$g_0\colon M\otimes\RR\to\RR^{p,q}$ and a base point~$z_0=g_0^{-1}(\RR^{0,q})$ of~$\Gr(M)$, where~$\RR^{0,q}$ is the span of the last~$q$ vectors of the standard basis of~$\RR^{p,q}$.
	The coordinates of~$\RR^{p,q}$ are denoted by~$x_j$.
	
	We recall from~\cite[Section~$3$]{bo;grass} the standard operators
	\bes
	\Delta=\sum_{j=1}^{p+q}\frac{\partial^2}{\partial x_j^2}\qquad\text{and}\qquad\exp\Big(-\frac{\Delta}{8\pi y}\Big)=\sum_{m=0}^\infty\frac{1}{m!}\Big(-\frac{\Delta}{8\pi y}\Big)^m
	\ees
	acting on smooth functions defined on~$\RR^{p,q}$.
	Note that~$\Delta$ is the Laplacian of~$\RR^{p+q}$ and not the one of~$\RR^{p,q}$.
	\begin{defi}\label{def;deg1genSiegth}
	Let~$\pol$ be a homogeneous polynomial on $\RR^{p,q}$ of degree $(m^+,m^-)$, i.e.\ homogeneous of degree $m^+$ in the first~$p$ variables, and homogeneous of degree~$m^-$ in the last~$q$ variables.
	The Siegel theta function $\Theta_M$ associated to~$M$ and~$\pol$ is defined as
	\ba\label{eq;bormegagentheta}
	\Theta_M(\tau,\boralpha,\borbeta,g,\pol)
	&=
	\sum_{\lambda\in M}\exp(-\Delta /8\pi y)(\pol)\big(g_0\circ g(\lambda+\borbeta)\big)\\
	&\quad
	\times e\Big(\tau q\big((\lambda+\borbeta)_{z^\perp}\big)+\bar{\tau}q\big((\lambda+\borbeta)_z\big)-(\lambda+\borbeta/2,\boralpha)\Big),
	\ea
	for every $\tau=x+iy\in\HH$, $\boralpha,\borbeta\in M\otimes\RR$, and~$g\in G$, where~$z=g^{-1}(z_0)\in\Gr(M)$.
	
	If~$\delta=\nu=0$, then we drop them from the notation and write only~$\Theta_M(\tau,g,\pol)$.
	\end{defi}
	\begin{rem}\label{rem;harmpol}
	If the polynomial~$\pol$ is \emph{harmonic}, i.e.~$\Delta\pol=0$, then~${\exp\big(-\Delta/8\pi y\big)(\pol)=\pol}$.
	This is the case of the polynomials~$\polab$ defined in~\eqref{eq;comppolPab}, if~$\alpha\neq\beta$.
	Instead, the polynomials~$\pol_{(\alpha,\alpha)}$ are homogeneous but non-harmonic, for every~$\alpha$.
	\end{rem}
	The theta function~$\Theta_M$ may be regarded as a non-holomorphic modular form with respect to the variable~$\tau\in\HH$.
	This is~\cite[Theorem~4.1]{bo;grass}, which we state in our notation as follows.
	\begin{thm}[Borcherds]\label{thm;borchmodtransftheta4}
	Let $M$ be a unimodular lattice of signature $(b^+,b^-)$.
	If~$\pol$ is a homogeneous polynomial of degree $(m^+,m^-)$ on~$\RR^{b^+,b^-}$, then
	\bas
	\Theta_M(\gamma\cdot\tau, a\boralpha+b\borbeta,c\boralpha+d\borbeta,g,\pol)=(c\tau+d)^{b^+/2+m^+}(c\bar{\tau}+d)^{b^-/2+m^-}\Theta_M(\tau,\boralpha,\borbeta,g,\pol),
	\eas
	for every $\gamma=\left(\begin{smallmatrix}
	a & b\\ c & d
	\end{smallmatrix}\right)\in\SL_2(\ZZ)$.
	\end{thm}
	
	In the remaining part of this section, we illustrate how to rewrite the Siegel theta function~$\Theta_M$ with respect to the split of some hyperbolic plane of~$M$, following the wording of~\cite[Section~5]{bo;grass}.
	This will be essential to unfold the theta integrals defining the Kudla--Millson lift.
	Since we do not need to use Borcherds' formalism in all its generality, from now on we restrict ourselves to the case of an even unimodular lattice~$L$ of signature~$(b,2)$ in place of~$M$.
	
	It is well known that any unimodular lattice~$L$ as above splits (up to isomorphisms) into an orthogonal direct sum of sublattices as
	\be\label{eq;splitofL}
	L= \underbrace{E_8\oplus\dots\oplus E_8\oplus U}_{=\brK}\oplus\, U,
	\ee
	where~$E_8$ is the~$8$-th root lattice and~$U$ is the hyperbolic lattice of rank~$2$.
	Let~$\brK$ be a Lorentzian unimodular sublattice of~$L$ defined as the orthogonal complement of some hyperbolic plane~$U$, as in~\eqref{eq;splitofL}.
	We may assume that the orthogonal basis~$(\basevec_j)_j$ of~$L\otimes\RR$ chosen above is such that~$\brK\otimes\RR$ is generated by~${\basevec_1,\dots,\basevec_{b-1},\basevec_{b+1}}$, and that~$U\otimes\RR$ is generated by~$\basevec_b$ and~$\basevec_{b+2}$.
	
	Let~$\genU, \genUU$ be a basis of~$U$ such that~$(\genU,\genU)=(\genUU,\genUU)=0$ and~$(\genU,\genUU)=1$.
	We may suppose that
	\be\label{eq;choiceofuu'}
	\genU=\frac{\basevec_b+\basevec_{b+2}}{\sqrt{2}}\qquad\text{and}\qquad\genUU=\frac{\basevec_b-\basevec_{b+2}}{\sqrt{2}}.
	\ee
	In this way, we may rewrite~$L$ as the orthogonal direct sum of~$\brK$ with~$\ZZ u\oplus \ZZ u'$.
	
	\begin{defi}\label{defi;not&deffromborw}
	Let~$z\in \Gr(L)$, and let~$g\in G$ be an isometry mapping~$z$ to~$z_0$. 
	We denote by~$w$ the orthogonal complement of $\genU_z$ in $z$, and by $w^\perp$ the orthogonal complement of $\genU_{z^\perp}$ in $z^\perp$.
	We denote by~$\borw \colon L\otimes\RR\to L\otimes\RR$ the linear map defined as~$\borw(\genvec)=g(\genvec_{w^\perp}+\genvec_w)$.
	\end{defi}
	We remark that~$\borw$ is an isometry from $w^\perp\oplus w$ to its image, and that it vanishes on~$\RR\genU_{z^\perp}\oplus\RR\genU_z$.
	
	\begin{defi}\label{def;bordefimplpol}
	Let $z\in \Gr(L)$, and let $g\in G$ be an isometry mapping~$z$ to~$z_0$.
	For every homogeneous polynomial $\pol$ of degree $(m^+,m^-)$ on~$\RR^{b,2}$, we define the homogeneous polynomials~$\pol_{\borw,h^+,h^-}$, of degrees respectively $(m^+-h^+,m^--h^-)$ on~${g_0\circ\borw(L\otimes\RR)\cong\RR^{b-1,1}}$, by
	\be\label{eq;bordefimplpol}
	\pol\big(g_0\circ g(\genvec)\big)=\sum_{h^+,h^-}(\genvec,\genU_{z^\perp})^{h^+}\cdot(\genvec,\genU_z)^{h^-}\cdot\polw{\borw}{h^+}{h^-}\big(g_0\circ\borw(\genvec)\big).
	\ee
	\end{defi}

	In Section~\ref{sec;follBorcherds} we will rewrite the Kudla--Millson theta form in terms of Siegel theta functions associated to the homogeneous polynomials~$\polab$ defined on~$\RR^{b,2}$ as in~\eqref{eq;comppolPab}, namely
	\bes
	\polab\big((x_1,\dots,x_{b+2})^t\big)=2x_\alpha x_\beta,
	\ees
	where the indices~$\alpha$ and~$\beta$ are in~$\{1,\dots,b\}$.
	Since the polynomials~$\polab$ are homogeneous of degree $(2,0)$, we may simplify~\eqref{eq;bordefimplpol} as
	\be\label{eq;decpolborh-0}
	\polab\big(g_0\circ g(\genvec)\big)=\sum_{h^+=0}^2(\genvec,\genU_{z^\perp})^{h^+}\cdot\polw{(\alpha,\beta),\borw}{h^+}{0}\big(g_0\circ\borw(\genvec)\big).
	\ee
	
	The following result provides a formula to compute~$\polw{(\alpha,\beta),\borw}{h^+}{0}$.
	\begin{lemma}\label{lemma;rewritborchimpldecourpol}
	For every~$g\in G$, the polynomial~$\polw{(\alpha,\beta),\borw}{h^+}{0}$ arising from the decomposition~\eqref{eq;decpolborh-0} of~$\polab$ may be computed as
	\ba\label{eq;closefforborpor}
	&\polw{(\alpha,\beta),\borw}{h^+}{0}\big(g_0\circ\borw(\genvec)\big)
	\\
	&\qquad
	=\begin{cases}
	\frac{2}{\genU_{z^\perp}^4} \big(g(\genU),\basevec_\alpha\big) \big(g(\genU),\basevec_\beta\big), & \text{if $h^+=2$,}\\
	\frac{2}{\genU_{z^\perp}^2} \big(g(\genU),\basevec_\alpha\big) \big(\borw(\genvec),\basevec_\beta\big) +\frac{2}{\genU_{z^\perp}^2} \big(g(\genU),\basevec_\beta\big) \big(\borw(\genvec),\basevec_\alpha\big), & \text{if $h^+=1$,}\\
	2 \big(\borw(\genvec),\basevec_\alpha\big)\big(\borw(\genvec),\basevec_\beta\big), & \text{if $h^+=0$,}
	\end{cases}
	\ea
	where~$z=g^{-1}(z_0)\in\Gr(L)$.
	\end{lemma}
	\begin{proof}
	For every $\genvec\in L\otimes\RR$, we denote by $x_j$ the coordinate of $\genvec$ with respect to the standard basis $\basevec_1,\dots,\basevec_{b+2}$ of $L\otimes\RR$.
	We recall that
	\bes
	\polab\big(g_0(\genvec)\big)=2x_\alpha x_\beta = 2(\genvec,\basevec_\alpha)(\genvec,\basevec_\beta).
	\ees
	If $g\in G=\GG(L\otimes\RR)$, then $\polab\big( g_0\circ g(\genvec)\big)=2\big(\genvec,g^{-1}(\basevec_\alpha)\big)\big(\genvec,g^{-1}(\basevec_\beta)\big)$.
	To rewrite the latter polynomial as in~\eqref{eq;decpolborh-0}, we rewrite $\big(\genvec,g^{-1}(\basevec_j)\big)$ in terms of $(\genvec,\genU_{z^\perp})$, for $j=\alpha,\beta$.
	
	The negative definite plane $z=g^{-1}(z_0)$ is generated by $g^{-1}(\basevec_{b+1})$ and $g^{-1}(\basevec_{b+2})$, whereas the positive definite $b$-dimensional subspace $z^\perp$ is generated by $g^{-1}(\basevec_1),\dots,g^{-1}(\basevec_b)$.
	Hence, the vectors $g^{-1}(\basevec_\alpha)$ and $g^{-1}(\basevec_\beta)$ lie in $z^\perp$.
	Recall that $w$ (resp.\ $w^\perp$) is the orthogonal complement of $u_z$ (resp.\ $u_{z^\perp}$) in $z$ (resp.\ $z^\perp$).
	We may decompose
	\be\label{eq;replprimstdbas}
	g^{-1}(\basevec_j)=s_j u_{z^\perp} + v_j',
	\ee
	for some $s_j\in\RR$, where $v_j'$ is the orthogonal projection of $g^{-1}(\basevec_j)$ to $w^\perp$ and~$j=\alpha,\beta$.
	
	By virtue of~\eqref{eq;replprimstdbas} we may rewrite~$\polab\big( g_0\circ g(\genvec)\big)$ as
	\be\label{eq;decompolabourcase}
	\polab\big( g_0\circ g(\genvec)\big)=2(\genvec,u_{z^\perp})^2s_\alpha s_\beta + (\genvec,u_{z^\perp})\big(
	2s_\alpha (\genvec,\genvec_\beta') + 2s_\beta (\genvec,\genvec_\alpha')\big)
	+ 2(\genvec,\genvec_\alpha')(\genvec,\genvec_\beta').
	\ee
	Comparing~\eqref{eq;decompolabourcase} with~\eqref{eq;decpolborh-0}, we deduce that
	\bes
	\polw{(\alpha,\beta),\borw}{h^+}{0}\big(g_0\circ\borw(\genvec)\big)=\begin{cases}
	2s_\alpha s_\beta, & \text{if $h^+=2$,}\\
	2s_\alpha (\genvec,\genvec_\beta') + 2s_\beta (\genvec,\genvec_\alpha'), & \text{if $h^+=1$,}\\
	2(\genvec,\genvec_\alpha')(\genvec,\genvec_\beta'), & \text{if $h^+=0$.}
	\end{cases}
	\ees
	Since $u_{z^\perp}$ is orthogonal to $w^\perp$, it follows that
	\bes
	s_j=\frac{\big(\genU_{z^\perp},g^{-1}(\basevec_j)\big)}{\genU_{z^\perp}^2}=
	\frac{\big(g(\genU),\basevec_j\big)}{\genU_{z^\perp}^2}.
	\ees
	Moreover, since~$\basevec_j$ is orthogonal to~$g(\genvec_w)$ for every~$j\le b$, we may rewrite
	\bes
	(\genvec,\genvec_j')=\big(\genvec_{w^\perp},g^{-1}(\basevec_j)\big)=\big(\borw(\genvec),\basevec_j\big).\qedhere
	\ees
	\end{proof}
	The following result illustrates how to decompose the Siegel theta function~$\Theta_L$ attached to the polynomial~$\polab$ with respect to the splitting~${L=\brK\oplus U}$ chosen in~\eqref{eq;splitofL}.
	It is~\cite[Theorem~5.2]{bo;grass}, rewritten with respect to our unimodular lattice~$L$.
	
	\begin{thm}[Borcherds]\label{thm;fromborchspltheta}
	Let $L=\brK\oplus U$ be a unimodular lattice of signature~$(b,2)$, and let~$\mu\in(\brK\otimes\RR)\oplus\RR u$ be the vector defined as
	\bes
	\mu=-\genUU+\genU_{z^\perp}/2\genU_{z^\perp}^2+\genU_z/2\genU_z^2.
	\ees
	We have
	\ba\label{eq;formbtfforpolab}
	&\Theta_L(\tau, g,\polab)
	\\
	&\quad=\frac{1}{\sqrt{2 y \genU_{z^\perp}^2}}\Theta_{\brK}(\tau,\borw,\polw{(\alpha,\beta),\borw}{0}{0})
	+ \frac{1}{\sqrt{2y\genU_{z^\perp}^2}}
	\sum_{\substack{c,d\in\ZZ\\ \gcd(c,d)=1}}
	\sum_{r\ge1}
	\sum_{h^+=0}^2
	\Big(-\frac{r}{2iy}\Big)^{h^+}
	\\
	&\quad\quad\times (c\bar{\tau}+d)^{h^+}e\Big(
	-\frac{r^2|c\tau+d|^2}{4iy\genU_{z^\perp}^2}
	\Big)\Theta_{\brK}\big(\tau,rd\mu,-rc\mu,\borw,\polw{(\alpha,\beta),\borw}{h^+}{0}\big).
	\ea
	\end{thm}
	\begin{rem}\label{rem;abofnotfrombrmu}
	In Theorem~\ref{thm;fromborchspltheta} we should write~$\mu_{\brK}$ as argument of~$\Theta_{\brK}$, namely the orthogonal projection of~$\mu$ to~$\brK\otimes\RR$, instead of~$\mu$.
	However, since~${\mu_{\brK}=\mu-(\mu,\genUU)\genU}$, we have
	\bas
	\mu_w&=(\mu_{\brK})_w=-\genUU_w,\\
	\mu_{w^\perp}&=(\mu_{\brK})_{w^\perp}=-\genUU_{w^\perp},\\
	(\mu,\genU)&=(\mu_{\brK},\genU).
	\eas
	This explain why we may use such an abuse of notation.
	Note also that the orthogonal projection~${L\otimes\RR\to\brK\otimes\RR}$ induces an isometric isomorphism~$w^\perp\oplus w\to w_{\Lor}^\perp\oplus w_{\Lor}=\brK\otimes\RR$.
	This implies that we may identify~$w$ with~$w_\Lor$ and consider~$w$ as an element of~$\Gr(\brK)$; see~\cite[p.~$42$]{br;borchp}.
	Analogously, we may regard~$\borw|_{\brK\otimes\RR}$ as an element of~$\GG(\brK\otimes\RR)$.
	\end{rem}
	
	\subsection{The Kudla--Millson theta form in terms of Siegel theta functions.}\label{sec;follBorcherds}
	In this section we explain how to rewrite the Kudla--Millson theta form~$\Theta(\tau,z,\varphi_{\text{\rm KM}})$ in terms of certain Siegel theta functions~$\Theta_L$.
	We then rewrite the latter with respect to a splitting~$L=\brK\oplus U$, for some Lorentzian lattice~$\brK$ and some hyperbolic plane~$U$.\\
	
	
	At the beginning of Section~\ref{sec;KMthetagen1} we rewrote the Kudla--Millson theta form as
	\ba\label{eq;KMdeg1recallexbis}
	\Theta(\tau,z,\varphi_{\text{\rm KM}})=\sum_{\alpha,\beta=1}^b \FFab(\tau,g) \cdot g^*(\omega_{\alpha,b+1}\wedge\omega_{\beta,b+2}),
	\ea
	in terms of certain auxiliary functions~$\FFab$ arising from the Schrödinger model~$\omega_\infty$ applied to	 the polynomial~$\Gpol_{(\alpha,\beta)}$ multiplied with the standard Gaussian~$\varphi_0$; see~\eqref{eq;KMdeg1recallexpl}.
	
	We are now ready to compute the auxiliary functions~$\FFab$ in terms of Siegel theta functions.
	\begin{lemma}\label{lemma;FabwithBorch}
	For every index~$\alpha,\beta=1,\dots,b$, we may rewrite the auxiliary function~$\FFab$ in terms of Siegel theta functions as
	\be\label{eq;FabwithBorch}
	\FFab(\tau,g)=
	y\cdot\Theta_L(\tau,g,\polab),
	\ee
	where~$\tau=x+iy\in\HH$,~$g\in G$, and~$\polab$ is the homogeneous polynomial defined in~\eqref{eq;comppolPab}.
	\end{lemma}
	\begin{proof}
	Suppose that~$\alpha\neq\beta$.
	Let $g_\tau=\big(\begin{smallmatrix} 1 & x\\ 0 & 1 \end{smallmatrix}\big)\left(\begin{smallmatrix} \sqrt{y} & 0\\ 0 & \sqrt{y}^{-1}\end{smallmatrix}\right)$ be the standard element of~$\SL_2(\RR)$ mapping~$i$ to~$\tau=x+iy$.
	Recall the Schrödinger model~$\omega_\infty$ of the Weil representation from Definition~\ref{defi;schrmod}.
	Since the polynomial~$\Gpol_{(\alpha,\beta)}=\polab$ is homogeneous of degree~$(2,0)$ on~$\RR^{b,2}$, we may compute
	\begin{align}\label{eq;inizFabwithBorch}
	\begin{split}\omega_\infty(g_\tau)\big(\polab\varphi_0\big)&\big(g_0\circ g(\genvec)\big)=y^{k/2}\cdot\omega_\infty\left(\begin{smallmatrix} 1 & x\\ 0 & 1\end{smallmatrix}\right)\big(\polab \varphi_0\big)\big(g_0\circ g(y^{1/2}\genvec)\big)\\
	&=y^{k/2} \cdot e\big(xq(\genvec)\big)\cdot\big(\polab\varphi_0\big)\big(g_0\circ g(y^{1/2}\genvec)\big)
	\\
	&=y^{1+k/2}\cdot e\big(xq(\genvec)\big)\cdot\polab\big(g_0\circ g(\genvec)\big)\cdot \varphi_0\big(g_0\circ g(y^{1/2}\genvec)\big).
	\end{split}
	\end{align}
	Since~$\varphi_0\big(g_0\circ g(y^{1/2}\genvec)\big)=e^{-\pi y(\genvec,\genvec)_z}$, where~$z=g^{-1}(z_0)\in\Gr(L)$, we may deduce that
	\bas
	e\big(x q(\genvec)\big)\cdot\varphi_0\big(g_0\circ g(y^{1/2}\genvec)\big)
	=
	e\big(x q(\genvec)\big)\cdot e^{-\pi y(\genvec,\genvec)_z}
	=
	e\big(\tau q(\genvec_{z^\perp})+\bar{\tau}q(\genvec_z)\big),
	\eas
	for every $\tau\in\HH$.
	This, together with~\eqref{eq;inizFabwithBorch}, implies that
	\bes
	\omega_\infty(g_\tau)\big(\polab\varphi_0\big)\big(g_0\circ g(\genvec)\big)=
	y^{1+k/2}\cdot\polab\big(g_0\circ g(\genvec)\big)\cdot e\big(\tau q(\genvec_{z^\perp})+\bar{\tau}q(\genvec_z)\big),
	\ees
	which we may insert into the formula defining~$\FFab$ in~\eqref{eq;KMdeg1recallexpl}, obtaining that
	\bes
	\FFab(\tau,g)=y\cdot\sum_{\lambda\in L}\polab\big(g_0\circ g(\lambda)\big)\cdot e\Big(
	\tau q(\lambda_{z^\perp})+\bar{\tau}q(\lambda_z)
	\Big).
	\ees
	It is enough to compare this with~\eqref{eq;bormegagentheta}, to deduce~\eqref{eq;FabwithBorch}.
	In fact, the polynomial~$\polab$ is harmonic; see Remark~\ref{rem;harmpol}.
	
	The case~$\alpha=\beta$ is analogous.
	The only difference is that
	\bes
	\Gpol_{(\alpha,\alpha)}(g_0\circ g(y^{1/2}\genvec)\big)=y\cdot\exp\big(-\Delta/8\pi y\big)(\pol_{(\alpha,\alpha)})\big(g_0\circ g(\genvec)\big).\qedhere
	\ees
	\end{proof}
%
	We conclude this section illustrating how to rewrite~$\FFab$ with respect to the splitting~${L=\brK\oplus U}$.
	\begin{cor}\label{cor;fromborchspltheta}
	For every~$\alpha,\beta=1,\dots,b$, we may rewrite the auxiliary function $\FFab(\tau,g)$ with respect to the splitting $L=\brK\oplus U$ as
	\ba\label{eq;splitFabviabor}
	\FFab(\tau,g)&=\frac{\sqrt{y}}{\sqrt{2 \genU_{z^\perp}^2}}\Theta_{\brK}(\tau,\borw,\polw{(\alpha,\beta),\borw}{0}{0}) +
	\frac{\sqrt{y}}{\sqrt{2\genU_{z^\perp}^2}}\sum_{\substack{c,d\in\ZZ\\ \gcd(c,d)=1}}
	\sum_{r\ge1}
	\sum_{h^+=0}^2
	\Big(-\frac{r}{2iy}\Big)^{h^+}\\
	&\quad\times (c\bar{\tau}+d)^{h^+}\cdot \exp\Big(
	-\frac{\pi r^2|c\tau+d|^2}{2y\genU_{z^\perp}^2}
	\Big)\cdot\Theta_{\brK}\big(\tau,rd\mu,-rc\mu,\borw,\polw{(\alpha,\beta),\borw}{h^+}{0}\big).
	\ea
	\end{cor}
	\begin{proof}
	It is a direct consequence of Lemma~\ref{lemma;FabwithBorch} and Theorem~\ref{thm;fromborchspltheta}.
	\end{proof}
		
	\section{Fourier expansions of complex-valued functions on~$\GG(L\otimes\RR)$}\label{sec;FexpLlorinvfuncts}
	
	In this section we recall two different models of~$\Gr(L)$, namely the \emph{projective model}~$\projmod$ in~$\PP(L\otimes\CC)$, and the \emph{tube domain model}~$\tubedom$ in~$\brK\otimes\CC$.
	We then explain how to identify the group of isometries~$G=\GG(L\otimes\RR)$ with the product~$K\times\tubedom$, where~$K$ is the stabilizer in~$G$ of the base point~$z_0\in\Gr(L)$.
	Furthermore, we illustrate how to use such an identification to construct Fourier expansions of complex-valued functions defined on~$G$ which are invariant with respect to translations by elements of~$\brK$.
	This will be relevant in Section~\ref{sec;unftr}, where we compute Fourier expansions of certain~$\brK$-invariant functions arising from a decomposition of the Kudla--Millson theta lift; see Theorem~\ref{thm;Fourexpofcoef}.
	\textcolor{\myblack}{Since the results appearing in this section are known to experts, we provide here only a quick overview}.
	
	Recall that we denote by~$L$ an even unimodular lattice of signature~$(b,2)$, by~$(\basevec_j)_j$ the standard basis of $L\otimes\RR$, and by~$\genU$ and~$\genUU$ the isotropic lattice vectors defined as in~\eqref{eq;choiceofuu'}.
	The lattice~$\brK$ is the Lorentzian sublattice of~$L$ orthogonal to the hyperbolic lattice~$\ZZ\genU\oplus\ZZ\genUU$.
	
	\subsection{Models of the symmetric space associated to~$\boldsymbol{G}$}\label{sec;modelsofsp}
	We denote by $\mathcal{D}_b$ the~$b$-dimensional complex manifold
	\bes
	\mathcal{D}_b=\big\{[Z_L]\in\PP(L\otimes\CC)\,:\,\text{$(Z_L,Z_L)=0$ and $(Z_L,\overline{Z_L})<0$}\big\}.
	\ees
	It has two connected components.
	We choose the one containing $[Z_L^0]$, where $Z_L^0\coloneqq[\basevec_{b+1}+i\basevec_{b+2}]$, and denote it by~$\projmod$.
	Such a component is identified with~$\Gr(L)$ \textcolor{\myblack}{as illustrated in~\cite[Part~II, Lemma~2.17]{1-2-3}}, explaining why $\projmod$ is usually referred as the \emph{projective model} of~$\Gr(L)$.
	\textcolor{\myblack}{The idea is to associated to~$z\in\Gr(L)$ an element~$[Z_L]\in\projmod$ with~$Z_L=X_L+iY_L$ by choosing a suitable basis~$X_L$,~$Y_L$ of the negative definite plane~$z$.}
	
	We now recall the tube domain model of $\Gr(L)$.
	If $Z_L\in L\otimes\CC$, then $Z_L=Z+au'+bu$ for some $Z\in\brK\otimes\CC$ and some~$a,b\in\CC$.
	We write~$Z_L=(Z,a,b)$ in short.
	The \emph{tube domain model}~$\tubedom$ is defined as the connected component of
	\bes
	\big\{Z=X+iY\in\brK\otimes\CC\,:\,Y^2<0\big\}
	\ees
	mapping to~$\projmod$ under the \textcolor{\myblack}{biholomorphism}
	\bes
	\tubedom \longrightarrow \projmod,\qquad Z\longmapsto[Z_L]=[(Z,1,-q(Z))].
	\ees
	\textcolor{\myblack}{We remark that for fixed~$[Z_L]\in\projmod$ there is a unique representative~$Z_L=(Z,1,-q(Z))$ such that~$Z=X+iY\in\tubedom$ and
	\be\label{eq;choiceforXlYl}
	X_L=(X,1,q(Y)-q(X))\qquad\text{and}\qquad Y_L=(Y,0,-(X,Y)).
	\ee	
	That representative depends on the choice of the isotropic vectors~$\genU$ and~$\genUU$.}
	The representative of the form $(Z_0,1,-q(Z_0))$ of the base point in~$\projmod$ is the one such that $Z_0=X_0+iY_0$, with~$X_0=0$ and~${Y_0=\sqrt{2}\basevec_{b+1}}$.
	
	We identified $\Gr(L)$ with $\projmod$ and $\tubedom$.
	The base point~${z_0=\langle \basevec_{b+1},\basevec_{b+2} \rangle}$ of~$\Gr(L)$ maps under such identifications respectively to
	\bas
	z_0\longleftrightarrow [Z_L^0]=[-\sqrt{2}\basevec_{b+2}+i\sqrt{2}\basevec_{b+1}] & \longleftrightarrow Z_0=i\sqrt{2}\basevec_{b+1}.
	\eas
	
	The following result can be regarded as a dictionary to rewrite functions defined on one of the previous models as functions on the remaining ones.
	It will be useful in Section~\ref{sec;compFexpafterunfold}, where we rewrite certain theta integrals in terms of the tube domain model~$\tubedom$.
	\begin{lemma}\label{lemma;someformufortubdom}
	Let $w$ (resp.\ $w^\perp$) be the orthogonal complement of $\genU_z$ (resp.\ $\genU_{z^\perp}$) in~$z$ (resp.~$z^\perp$), and let $\mu=-\genUU+\genU_{z^\perp}/2\genU_{z^\perp}^2+\genU_z/2\genU_z^2$.
	If $Z=X+iY\in\tubedom$ corresponds to~$z$ under the previous identifications, and if the representative of the corresponding point~${[Z_L]=[X_L+iY_L]}$ in $\projmod$ is chosen such that~\eqref{eq;choiceforXlYl} is satisfied, then
	\ba\label{eq;someformufortubdom}
	\begin{split}
	X_L^2 & = Y_L^2 = Y^2\\
	\genU_{z^\perp}^2 & =-\genU_{z}^2=-1/Y^2,\\
		{\lambda}_w & =(\lambda,Y)Y/Y^2,
	\end{split}
	\qquad
	\begin{split}
	\genU_z & = X_L/Y^2\\
	\mu_{\brK} & =X,\\
		(\lambda,\lambda)_w &
	=\lambda^2-2(\lambda,Y)^2/Y^2,
	\end{split}
	\ea
	where~$\lambda$ is any vector of~$\brK\otimes\RR$, and~$\mu_{\brK}$ is the orthogonal projection of~$\mu$ to~$\brK\otimes\RR$.
	\end{lemma}
	\begin{proof}
	See e.g.~\cite[p.~$543$]{bo;grass} and~\cite[pp.~$79, 80$]{br;borchp}.
	\end{proof}
	
	\subsection{The identification of~$\boldsymbol{K\times\tubedom}$ with~$\boldsymbol{G}$}\label{sec;theidentof}
	Let $z\in\Gr(L)$, and let $Z=X+iY\in\tubedom$ and $[Z_L]\in\projmod$ be the corresponding points in the tube domain model and in the projective model, respectively.
	From now on we suppose that $Z_L=X_L+iY_L$ is the only representative of~$[Z_L]$ such that~\eqref{eq;choiceforXlYl} is fulfilled.
	
	We want to fix \emph{once and for all} an identification of $K\times\tubedom$ with~$G$, i.e.\ a diffeomorphism
	\be\label{eq;finident}
	\iden \colon K\times\tubedom \longrightarrow G.
	\ee
	For the purposes of this article, we need to choose an identification~$\iden$ fulfilling the properties illustrated in the following result.
	The reason, which will become clear with Theorem~\ref{thm;Fourexpofcoef}, is that we need to use such properties to prove that some series, arising from the Kudla--Millson lift, are actually Fourier expansions of complex-valued functions defined on~$G$.
	\textcolor{\myblack}{Recall the construction of~$w$,~$w^\perp$ and~$\borw$ from Definition~\ref{defi;not&deffromborw}.}
	
	\begin{lemma}\label{lemma;choosegoodiden}
	There exists a diffeomorphism~$\iden\colon K\times\tubedom\to G$ such that
	\bas
	&\iden(\kappa,Z)=\kappa\cdot\iden(1,Z),\qquad	&&\iden(1,Z)\colon z\longmapsto z_0,\qquad\text{and} &&&\iden(1,Z)\colon \RR\genU\longmapsto \RR\genU,\\
	\eas
	and also such that the associated function~$\borwiden{1}{Z}|_{\brK\otimes\RR}$ does not depend on the real part of~$Z$, or equivalently~\textcolor{\myblack}{$\borwiden{1}{Z}(\genvec)=\borwiden{1}{Z+X'}(\genvec)$ for every~$\genvec,X'\in\brK\otimes\RR$.}
	\end{lemma}
%
%
%
%
	\textcolor{\myblack}{To construct such a~$\iden$ we use the following explicit Iwasawa decomposition of~$G$.}
	We choose a basis of~$L\otimes\RR$ which differs both from the orthonormal one used to construct the Kudla--Millson Schwartz function, that we denoted by~$(\basevec_j)_j$, and from the one used in~\cite[Section~$4.1$]{br;borchp} to give coordinates of~$\tubedom$.
	The reason is that the new basis enables us to rewrite the factors~$A$ and $N$ of an Iwasawa decomposition~$G=KAN$ as groups of matrices with an easy description, namely as diagonal matrices for the former, and upper triangular for the latter.
	
	The new basis we choose is the one given by
	\be\label{eq;newbasisKAN}
	\genU,\genUtwo,d_3,\dots,d_b,\genUUtwo,\genUU,
	\ee
	where $d_j\coloneqq \basevec_{j-2}$ for $3\le j\le b$, while~\textcolor{\myblack}{$\genUtwo\coloneqq (\basevec_{b-1} + \basevec_{b+1})/\sqrt{2}$ and~$\genUUtwo\coloneqq(\basevec_{b-1} - \basevec_{b+1})/\sqrt{2}$} are the standard generators of the hyperbolic plane~$U$ split off orthogonally by~$\brK$, such that~$\brK=\Lpos\oplus U$ for some positive definite unimodular lattice~$\Lpos$.
%
	
	As illustrated e.g.\ in~\cite[Section~$5.1$]{ohmorris} and~\cite[Section~$2.3$]{Livinsky}, we may realize the Iwasawa decomposition of~$G=\GG(L\otimes\RR)$ over the basis~\eqref{eq;newbasisKAN} as~$G=KAN$, where~$K$ is the stabilizer of the base point~$z_0=\langle \genU-\genUU,\genUtwo-\genUUtwo\rangle_\RR$, which is the same we chose in the previous sections,~\textcolor{\myblack}{$A=\big\{\diag(m_1,m_2,1,\dots,1,m_2^{-1},m_1^{-1}) : m_1,m_2\in\RR_{>0} \big\}$} is a group of diagonal matrices with non-negative entries, and~$N$
	is \textcolor{\myblack}{the group of upper triangular matrices as in~\cite[First paragraph of Section~$5.2$]{ohmorris}}.
	
	If~$Z=X+iY\in\brK\otimes\CC$, we may rewrite it with respect to the basis~\eqref{eq;newbasisKAN} as the column vector~$Z=(0,Z_2,\dots,Z_{b+1},0)^t$, for some~$Z_j\in\CC$, and analogously for the real and imaginary part of~$Z$.
	We may rewrite the tube domain model~$\tubedom$ with respect to the new basis~\eqref{eq;newbasisKAN} as
	\bes
	\tubedom=\{ Z\in\brK\otimes\CC : q(Y)<0\text{ and }Y_{b+1}<0 \}.
	\ees
	\textcolor{\myblack}{One can easily use the induced action of~$G$ on~$\tubedom$ to prove the following result.}
	\begin{lemma}\label{lemma;transltubedom}
	Let~$X'\in\brK\otimes\RR$, and let~$M(X')\in N$ be the matrix defined as
	\bes
	M(X')=\left(\begin{smallmatrix}
	1 & -X'_{b+1} & -X'_3 & \cdots & -X'_b & -X'_2 & -q(X')\\
	  & 1 & 0 & \cdots & \cdots & 0 & X'_2\\
	  &   & \ddots & \ddots & & \vdots & X'_3\\
	  & & & \ddots & \ddots & \vdots & \vdots \\
	  & & & & \ddots & 0 & X'_b\\
	  & & & & & 1 & X'_{b+1}\\
	  & & & & & & 1
	\end{smallmatrix}\right),
	\ees
	where we denote by~$X'_j\in\RR$ the $j$-th coordinate of~$X'$ with respect to the basis~\eqref{eq;newbasisKAN}.
	The action of~$M(X')$ on~$\tubedom$ is given by the translation~$Z\mapsto Z+X'$.
	\end{lemma}
%
	
	We recall that~$AN$ acts on~$\Gr(L)$ bijectively, that is, for every~$z\in\Gr(L)$ there exists only one~$a\in A$ and~$n\in N$ such that~$an$ maps~$z_0$ to~$z$.
	\begin{defi}
	Let~$G=KAN$ be the Iwasawa decomposition of~$G=\GG(L\otimes\RR)$ constructed above.
	If~$Z\in\tubedom$ corresponds to the negative definite plane~$z\in\Gr(L)$, then we define~$\iden(1,Z)\coloneqq (an)^{-1}$, where~$a\in A$ and~$n\in N$ are chosen such that~$an$ maps~$z_0$ to~$z$. 
	We also set~$\iden(\kappa,Z)=\kappa\cdot\iden(1,Z)$, for every~$\kappa\in K$.
	\end{defi}
	\textcolor{\myblack}{It is easy to check that the~$\iota$ above satisfies the properties illustrated in Lemma~\ref{lemma;choosegoodiden}.}

	\subsection{Fourier expansions}\label{Fexp}
	In this section we introduce Fourier expansions of $\brK$-invariant complex valued functions defined over~$G$.
	
	Recall that the sublattice~$\brK$ is unimodular.
	If a smooth function $F\colon \tubedom\to\CC$ is~$\brK$-invariant function, i.e.~$F(Z+\lambda)=F(Z)$ for every~$\lambda\in\brK$, then it admits a Fourier expansion of the form
	\bes
	F(Z)=\sum_{\lambda\in\brK}c(\lambda,Y)\cdot e\big((\lambda,X)\big),
	\ees
	where we denote by~$c(\lambda,Y)$ the Fourier coefficient of~$F$ associated to~$\lambda$ and~$Y$.

	It is possible to consider Fourier expansions of $\brK$-invariant functions defined over~$G$ instead of $\tubedom$, as we are going to illustrate.
	
	If $F\colon G\to\CC$ is a smooth function defined over~$G$, we may use an identification~$\iota$ as in Section~\ref{sec;theidentof} to rewrite~$F$ as a function of the form~$F\colon K\times\tubedom\to\CC$, which we denote with the same symbol.
	Suppose that $F$ is $\brK$-invariant, i.e.
	\bes
	F(\kappa,Z+\lambda)=F(\kappa,Z)\qquad\text{for every $\kappa\in K$, $Z\in\tubedom$ and $\lambda\in\brK$,}
	\ees
	then $F$ admits a Fourier expansion
	\be\label{eq;FexpgenfunctG}
	F(g)=F(\kappa,Z)=
	\sum_{\lambda\in\brK}c(\lambda,\kappa,Y)\cdot e\big((\lambda,X)\big),
	\ee
	where $g\in G$ is identified with $(\kappa,Z)\in\tubedom\times K$ under~$\iota$.
	The functions~$c(\lambda,\kappa,Y)$ are the \emph{Fourier coefficients (with respect to~$\iden$)} of~$F$.

	\section{The unfolding of the Kudla--Millson lift}\label{sec;theunfofKMgen1}
	
	In this section we explain how to compute the Kudla--Millson lift of a cusp form~$f$ in terms of integrals of~$f$ against certain Siegel theta functions.
	We then unfold such integrals and deduce their Fourier expansions as complex-valued functions defined on~$G$.
	
	Recall that~$k=1+b/2$.
	The genus~$1$ Kudla--Millson lift~${\KMliftbase\colon S^k_1\to\mathcal{Z}^2(\hermdom)}$ has already been introduced with Definition~\ref{def;KMliftgen1}.
	It produces~$\Gamma$-invariant~$2$-forms on~$\hermdom$ which descend to~$2$-forms on the orthogonal Shimura variety~$X_\Gamma=\Gamma\backslash\hermdom$, for every arithmetic subgroup~$\Gamma$ of~$\GG(L)$.
	
	By Lemma~\ref{lemma;FabwithBorch} and~\eqref{eq;KMdeg1recallexpl}, we may rewrite the Kudla--Millson lift of a cusp form~$f\in S^k_1$ in terms of Siegel theta functions as
	\be\label{eq;KMlambdaaO}
	\KMliftbase(f)=\sum_{\alpha,\beta=1}^{b}\Big(
	\underbrace{\int_{\SL_2(\ZZ)\backslash\HH}y^{k+1}f(\tau)\overline{\Theta_L(\tau,g,\polab)}\,\frac{dx\,dy}{y^2}}_{\eqqcolon\intfunct(g)}
	\Big)
	\cdot
	g^*\big(\omega_{\alpha, b+1}\wedge\omega_{\beta, b+2}\big),
	\ee
	for every~$g\in G$ mapping~$z$ to~$z_0$.
	The closed $2$-form~$\KMliftbase(f)$ at the point~$z\in\hermdom$ does not depend on the choice of such~$g$.
	
	We refer to the integrals~$\intfunct$ appearing in~\eqref{eq;KMlambdaaO} as the \emph{defining integrals} of~$\KMliftbase(f)$.
	They are complex-valued functions defined over~$G$.
	The goal of this section is to compute a Fourier expansion of such defining integrals~$\intfunct$ by means of the \emph{unfolding trick}.

	\subsection{The unfolding of $\boldsymbol{\KMliftbase}$}\label{sec;unftr}
	We aim to unfold the defining integrals~$\intfunct$ appearing in~\eqref{eq;KMlambdaaO} of the Kudla--Millson lift by means of the \textcolor{\myblack}{Rankin--Selberg method}.
	To do so, we follow the same strategy of Borcherds~\cite[Section~$5$]{bo;grass}, rewriting the Siegel theta functions with respect to a split~$L=\brK\oplus U$.
	Recall the polynomials~$\polw{(\alpha,\beta),\borw}{h^+}{0}$ from Lemma~\ref{lemma;rewritborchimpldecourpol}\textcolor{\myblack}{, and let~$\Gamma_\infty$ be the index~$2$ subgroup~$\{\left(\begin{smallmatrix}
	1 & n\\ 0 & 1
	\end{smallmatrix}\right) : n\in\ZZ\}$ of the group of translations in~$\SL_2(\ZZ)$}.
	\begin{prop}\label{prop;compauxhab}
	Let~$\alpha,\beta=1,\dots,b$ and let~$f\in S^k_1$. Denote by~$\hhab$ the~$\Gamma_\infty$-invariant auxiliary function defined as
	\bas
	\hhab(\tau,g)&=\frac{y^{k+1/2}f(\tau)}{\sqrt{2}\,|\genU_{z^\perp}|}\sum_{r\ge1} \sum_{h^+=0}^2
	\Big(\frac{r}{2iy}\Big)^{h^+}
	\exp\Big(
	-\frac{\pi r^2}{2y\genU_{z^\perp}^2}
	\Big)
	\\
	&\quad\times\overline{\Theta_{\brK}\big(\tau,r\mu,0,\borw,\polw{(\alpha,\beta),\borw}{h^+}{0}\big)},
	\eas
	for every~$\tau\in\HH$ and~$g\in G$, where~$z=g^{-1}(z_0)\in\Gr(L)$.
	The integrands appearing in the defining integrals~$\intfunct$ of the lift~$\KMliftbase(f)$ may be rewritten as
	\ba\label{eq;unfoldingtrickform}
	y^{k+1}f(\tau)\,\overline{\Theta_L(\tau,g,\polab)}
	&=
	\frac{y^{k+1/2}f(\tau)}{\sqrt{2}\,|\genU_{z^\perp}|}\,\overline{\Theta_{\brK}(\tau,\borw,\polw{(\alpha,\beta),\borw}{0}{0})}
	\\
	&\quad+
	\sum_{\gamma=\left(\begin{smallmatrix}
	* & *\\ c & d
	\end{smallmatrix}\right)\in\Gamma_\infty\backslash\SL_2(\ZZ)}\hhab(\gamma\cdot\tau,g).
	\ea
	\end{prop}
	\begin{proof}
	The definition of $\hhab$ above corresponds to the product of~$y^kf(\tau)$ with the conjugate of the summand in~\eqref{eq;splitFabviabor} associated to the values~$c=0$ and~$d=1$.
	Such a function is~$\Gamma_\infty$-invariant, since so is also~$\Theta_{\brK}(\tau,r\mu,0,\borw,\pol_{(\alpha,\beta),\borw,h^+,h^-})$.
	
	Let $\gamma=\big(\begin{smallmatrix}
	a & b\\ c & d
	\end{smallmatrix}\big)\in\Gamma_\infty\backslash\SL_2(\ZZ)$, for some coprime integers $c,d\in\ZZ$, and let $g\in G$.
	By the modular transformation properties of $y$, $f(\tau)$ and $\Theta_{\brK}$, where the automorphic factor of the latter is given by Theorem~\ref{thm;borchmodtransftheta4}, we deduce that
	\ba\label{eq;proofposunftr}
	\hhab(\gamma\cdot\tau,g)&=\frac{1}{\sqrt{2}\,|\genU_{z^\perp}|}\cdot \frac{y^{k+1/2}}{|c\tau+d|^{2k+1}}
	(c\tau+d)^k
	f(\tau)
	\\
	&\quad \times \sum_{h^+=0}^2 |c\tau+d|^{2h^+} \sum_{r\ge1}\Big(\frac{r}{2iy}\Big)^{h^+}
	\exp\Big(
	-\frac{\pi r^2|c\tau+d|^2}{2y\genU_{z^\perp}^2}
	\Big)
	\\
	&\quad \times (c\bar{\tau}+d)^{(b-1)/2+2-h^+}(c\tau+d)^{1/2}\,\overline{\Theta_{\brK}(\tau,M,N,\borw,\polw{(\alpha,\beta),\borw}{h^+}{0})},
	\ea
	where $M,N\in\brK\otimes\RR$ are such that $aM+bN=r\mu$ and $cM+dN=0$.
	The solutions of the previous system of equations are $M=rd\mu$ and $N=-rc\mu$.
	We replace them in~\eqref{eq;proofposunftr}, and simplify the factors given by powers of~$(c\tau+d)$ and their conjugates, deducing that
	\bas
	\hhab(\gamma\cdot\tau,g)&=\frac{y^{k+1/2}}{\sqrt{2}\,|\genU_{z^\perp}|}
	f(\tau)
	\sum_{r\ge1}
	\sum_{h^+=0}^2
	\Big(\frac{r}{2iy}\Big)^{h^+} (c\tau+d)^{h^+}
	\\
	& \quad\times
	\exp\Big(
	-\frac{\pi r^2|c\tau+d|^2}{2y\genU_{z^\perp}^2}
	\Big)\,\overline{\Theta_{\brK}(\tau,rd\mu,-rc\mu,\borw,\polw{(\alpha,\beta),\borw}{h^+}{0})}.
	\eas
	From this and Corollary~\ref{cor;fromborchspltheta} we deduce that the value~$\hhab(\gamma\cdot\tau,g)$ equals the~$(c,d)$-summand of~${y^kf(\tau)\,\overline{\FFab(\tau,g)}}$ arising when rewriting $\FFab(\tau,g)$ as in~\eqref{eq;splitFabviabor}.
	This concludes the proof of~\eqref{eq;unfoldingtrickform}.
	\end{proof}
	\begin{cor}\label{cor;unfolded!}
	Let~$f\in S^k_1$.
	We may unfold the defining integrals~$\intfunct$ of the Kudla--Millson lift~$\KMliftbase(f)$ as
	\ba\label{eq;firstunftrickap}
	\intfunct(g)
	&=
	\int_{\SL_2(\ZZ)\backslash\HH}\frac{y^{k+1/2}f(\tau)}{\sqrt{2}\,|\genU_{z^\perp}|}\,\overline{\Theta_{\brK}(\tau,\borw,\polw{(\alpha,\beta),\borw}{0}{0})}\frac{dx\,dy}{y^2}
	\\
	&\quad + 2\int_{\Gamma_\infty\backslash\HH}\hhab(\tau,g)\frac{dx\,dy}{y^2},
	\ea
	where~$\hhab$ is the auxiliary function provided by Proposition~\ref{prop;compauxhab}.
	\end{cor}
	\begin{proof}
	It is enough to apply the unfolding trick to the integral over~$\SL_2(\ZZ)\backslash\HH$ of the right-hand side of~\eqref{eq;unfoldingtrickform}.
	\end{proof}

	\subsection{Fourier expansions of unfolded integrals}\label{sec;compFexpafterunfold}	
	In this section we compute the Fourier expansion of the defining integrals~$\intfunct\colon G\to\CC$ of~$\KMliftbase(f)$ appearing in~\eqref{eq;KMlambdaaO}, for every~$f\in S^k_1$.
	To do so, we will replace in the unfolded integrals provided by Corollary~\ref{cor;unfolded!} the cusp form~$f$ with its Fourier expansion, and the Siegel theta function~$\Theta_{\brK}$ with its defining series.
	We write the Fourier expansion of~$f$ as
	\be\label{eq;fexpf}
	f(\tau)=\sum_{n>0}c_n(f)e(n\tau)=\sum_{n>0}c_n(f)\exp(-2\pi ny)e(nx),
	\ee
	where~$\tau= x + iy$.
	
	Recall that we denote by~$(\cdot{,}\cdot)_w$ the \emph{standard majorant} with respect to~${w\in\Gr(\brK)}$, that is~$(\genvec,\genvec)_w=\genvec_{w^\perp}^2 - \genvec_w^2$, for every~$\genvec\in \brK\otimes\RR$.
	We rewrite
	\ba\label{eq;precFexpofTllor}
	\Theta_{\brK}\big(\tau,r\mu,0,\borw,\polw{(\alpha,\beta),\borw}{h^+}{0}\big)
	&=
	\sum_{\lambda\in\brK}\exp(-\Delta/8\pi y)\big(\polw{(\alpha,\beta),\borw}{h^+}{0}\big)\big(g_0\circ\borw(\lambda)\big)
	\\
	&\quad\times \exp\big(-\pi y(\lambda,\lambda)_w\big)\cdot e\big(xq(\lambda)\big)\cdot e\big(-r(\lambda,\mu)\big),
	\ea
	with respect to the decomposition of~$\tau$ in real and imaginary part.
	\begin{rem}
	Even if $\polab$ is harmonic, namely if~$\alpha\neq\beta$, the polynomials~$\polw{(\alpha,\beta),\borw}{h^+}{0}$ may not be harmonic.
	If $h^+=1,2$, then they are of degree respectively~$0$ and~$1$, so they are harmonic.
	But the harmonicity of the one associated to~$h^+=0$ depends on the choice of~$g$, as illustrated in the following example.
	This explains why the operator~$\exp(-\Delta/8\pi y)$ appearing in~\eqref{eq;precFexpofTllor} can not be in general dropped, even under the hypothesis that~$\alpha\neq\beta$.
	\end{rem}
	\begin{ex}\label{ex;nonharmh+=2}
	We are going to construct an isometry~${g\in G=\GG(L\otimes\RR)}$ such that the polynomial~$\polw{(\alpha,\beta),\borw}{0}{0}$ is non-harmonic.
	
	Suppose that~$\alpha\neq\beta$ and that~$\alpha,\beta<b$.
	Let $g\in G$ be the isometry defined as
	\bes
	g\colon \basevec_\alpha\mapsto\frac{\basevec_\alpha+\basevec_\beta}{\sqrt{2}},\qquad \basevec_b\mapsto\frac{\basevec_\alpha-\basevec_\beta}{\sqrt{2}},\qquad \basevec_\beta\mapsto \basevec_b,
	\ees
	and fixing all remaining vectors of the standard basis of~$L\otimes\RR$.
	We remark that such an isometry lies in the maximal compact subgroup~$K$ of~$G$, that is, in the stabilizer of the base point~$z_0\in\Gr(L)$.
	
	Recall that~$\polab\big(g_0(\genvec)\big)=2x_\alpha x_\beta$, for every $\genvec=\sum_{j=1}^{b+2}x_j \basevec_j\in L\otimes\RR$.
	For the special choice of the isometry~$g$ as above, we may deduce that
	\be\label{eq;exnonharmh+0}
	\polab\big(g_0\circ g(\genvec)\big)=x_\alpha^2-x_b^2,
	\ee
	since
	\bes
	g(\genvec)=x_1\basevec_1+\dots+\Big(\frac{x_\alpha+x_b}{\sqrt{2}}\Big)\basevec_\alpha + \dots + \Big(\frac{x_\alpha-x_b}{\sqrt{2}}\Big)\basevec_\beta +\dots+x_\beta \basevec_b +\dots + x_{b+2}\basevec_{b+2}.
	\ees
	
	We are now ready to compute the polynomials~$\polw{(\alpha,\beta),\borw}{h^+}{0}$.
	Since~$u=(\basevec_b+\basevec_{b+2})/\sqrt{2}$, we deduce that~${u_{z_0^\perp}=\basevec_b/\sqrt{2}}$, hence~${(\genvec,u_{z_0^\perp})=x_b/\sqrt{2}}$.
	By comparing~\eqref{eq;exnonharmh+0} with~\eqref{eq;decpolborh-0}, or directly by Lemma~\ref{lemma;rewritborchimpldecourpol}, we deduce that
	\bes
	\polw{(\alpha,\beta),\borw}{h^+}{0}\big(g_0\circ \borw(\genvec)\big)=\begin{cases}
	x_\alpha^2, & \text{if $h^+=0$,}\\
	0, & \text{if $h^+=1$,}\\
	-2, & \text{if $h^+=2$.}
	\end{cases}
	\ees
	In particular, the polynomial $\polw{(\alpha,\beta),\borw}{0}{0}$ is non-harmonic.
	\end{ex}
	
	We are now ready to prove the main result of this section.
	Recall the defining integrals~$\intfunct$ of the Kudla--Millson lift of a cusp form from~\eqref{eq;KMlambdaaO}.
	\begin{thm}\label{thm;Fourexpofcoef}
	Let~$f\in S^k_1$.
	We identify~$G$ with~$K\times\tubedom$ under a diffeomorphism~$\iden$ as in Lemma~\ref{lemma;choosegoodiden}, such that every $g\in G$ may be rewritten as~$\iden(\kappa,Z)$, for a unique~$(\kappa,Z)\in K\times\tubedom$.
	The defining integrals~$\intfunct\colon G\to\CC$ of the Kudla--Millson lift~$\KMliftbase(f)$ have a Fourier expansion of the form
	\be\label{eq;foucoefcor0}
	\intfunct(g)=\intfunct\big(\iden(\kappa,Z)\big)=\sum_{\lambda\in\brK}c(\lambda,\kappa,Y)\cdot e\big((\lambda,X)\big)	,
	\ee
	where we decompose~$Z=X+iY\in\tubedom$.
	
	The Fourier coefficient of~$\intfunct$ associated to any~$\lambda\in\brK$ with~$q(\lambda)>0$ is
	\ba\label{eq;foucoefcor}
	c(\lambda,\kappa,Y)&=\frac{\sqrt{2}}{|\genU_{z^\perp}|}
	\sum_{h^+=0}^2
	\sum_{\substack{t\ge 1\\ t|\lambda}}
	\Big(\frac{t}{2i}\Big)^{h^+}
	c_{q(\lambda)/t^2}(f)
	\int_0^{+\infty} y^{k-h^+-3/2}
	\\
	&\quad\times\exp\Big(-\frac{2\pi y \lambda_{w^\perp}^2}{t^2} -
	\frac{\pi t^2}{2y\genU_{z^\perp}^2}
	\Big)\cdot\exp(-\Delta/8\pi y)\big(\polw{(\alpha,\beta),\borw}{h^+}{0}\big)\big(g_0\circ\borw(\lambda/t)\big)dy,
	\ea
	where we say that an integer $t\ge1$ divides $\lambda\in\brK$, in short~$t|\lambda$, if and only if~$\lambda/t$ is still a lattice vector in $\brK$.
	
	The Fourier coefficient of~$\intfunct$ associated to~$\lambda=0$, i.e. the constant term of the Fourier expansion, is
	\ba\label{eq;intconstcoef}
	c(0,\kappa,Y)=\int_{\SL_2(\ZZ)\backslash\HH}\frac{y^{k+1/2}f(\tau)}{\sqrt{2}\, |\genU_{z^\perp}|}
	\,
	\overline{\Theta_{\brK}(\tau,\borw,\polw{(\alpha,\beta),\borw}{0}{0})}
	\,
	\frac{dx\,dy}{y^2}.
	\ea
	
	In all remaining cases, the Fourier coefficients are trivial.
	\end{thm}
	Implicit in~\eqref{eq;foucoefcor} and~\eqref{eq;intconstcoef} is that the right-hand sides do not depend on~$X$.
	This is shown in the proof of Theorem~\ref{thm;Fourexpofcoef} using the following result.
	\begin{lemma}\label{lemma;nodependenceonXgen1}
	Let~$\pol$ be a homogeneous polynomial of degree~$(m^+,m^-)$ on~$\RR^{b,2}$.
	We identify~$G$ with~$K\times\tubedom$ under a diffeomorphism~$\iden$ as in Lemma~\ref{lemma;choosegoodiden}.
	The value of the function
	\bes
	\pol_{\borw,h^+,h^-}\big(g_0\circ\borw(\lambda)\big)
	\ees
	with respect to the variable~$g=\iden(\kappa,Z)\in G$ does not depend on the real part~$X$ of~$Z$, for any~$\lambda\in\brK\otimes\RR$ and any~$h^+,h^-$.
	\end{lemma}
	\begin{proof}[Proof of Lemma~\ref{lemma;nodependenceonXgen1}]
	Recall that we denote by~$x_j=(\genvec,\basevec_j)$ the coordinate of any vector~${\genvec\in L\otimes\RR}$ with respect to the standard basis vector~$\basevec_j$, and by~$g_0\colon L\otimes\RR\to\RR^{b,2}$ the isometry defined as~$g_0(\genvec)=(x_1,\dots,x_{b+2})^t$.
	If~$Z\in\tubedom$, we denote by~$z$ its corresponding point on the Grassmannian~$\Gr(L)$.
	
	By Lemma~\ref{lemma;choosegoodiden}, the isometry~$\iden(1,Z)$ preserves the isotropic line~$\RR\genU$, for every~$Z\in\tubedom$.
	This means that there exists a function~$c\colon\tubedom\to\RR\setminus\{0\}$ such that~$\iden(1,Z)(\genU)=c(Z)\cdot\genU$.
	Since~$\iden$ is a diffeomorphism, the function~$c$ is smooth.
	Moreover, since~$\iden(1,Z_0)$ is the identity by construction, and hence~$c(Z_0)=1$ where~$Z_0$ is the point of the tube domain identified with the base point~$z_0\in\Gr(L)$, then~$c(Z)>0$ for every~$Z\in\tubedom$.
	The vector~$\genU_z/|\genU_z|$ has norm~$1$, hence
	\bes
	\iden(1,Z)\Big(\frac{\genU_z}{|\genU_z|}\Big)=\frac{c(Z)}{|\genU_z|}\cdot\genU_{z_0}
	\ees
	is a norm~$1$ vector, from which we deduce that~$c(Z)=|\genU_z|/|\genU_{z_0}|=|\genU_{z^\perp}|/|\genU_{z^\perp_0}|$.
	
	For every~$g\in G$, we rewrite~$g^{-1}(\basevec_j)$ with respect to the decomposition
	\bes
	L\otimes\RR=\RR\genU_{z^\perp}\oplus\RR\genU_z\oplus w^\perp\oplus w
	\ees
	as
	\be\label{eq;splitg-1v_j}
	g^{-1}(\basevec_j)=A_j(g)\cdot\genU_{z^\perp} + B_j(g)\cdot\genU_z + g^{-1}(\basevec_j)_{w^\perp\oplus w},
	\ee
	where~$A_j,B_j\colon G\to\RR$ are the auxiliary functions defined as
	\bes
	A_j(g)=\frac{\big(g^{-1}(\basevec_j),\genU_{z^\perp}\big)}{\genU_{z^\perp}^2}\qquad\text{and}\qquad B_j(g)=\frac{\big(g^{-1}(\basevec_j),\genU_z\big)}{\genU_z^2},
	\ees
	and where~$g^{-1}(\basevec_j)_{w^\perp\oplus w}$ is the orthogonal projection of~$g^{-1}(\basevec_j)$ on~$w^\perp\oplus w$.
	Suppose that~$g=\iden(\kappa,Z)$, for some~$\kappa\in K$ and~$Z\in\tubedom$.
	We may compute
	\be\label{eq;A_jcomexplgen1}
	A_j\big(\iden(\kappa,Z)\big) = \frac{\Big(\basevec_j,\big(\kappa\cdot\iden(1,Z)(u)\big)_{z_0^\perp}\Big)}{\genU_{z^\perp}^2}=
	\frac{\Big(\basevec_j,c(Z)\cdot\big(\kappa(u)\big)_{z_0^\perp}\Big)}{\genU_{z^\perp}^2}=\frac{\Big(\basevec_j,\kappa(\genU_{z^\perp_0})\Big)}{|\genU_{z^\perp}|\cdot|\genU_{z_0^\perp}|}.
	\ee
	Since~$|\genU_{z^\perp}|=1/|Y|$ by Lemma~\ref{lemma;someformufortubdom}, we deduce that the value of the function~$A_j$ does not depend on~$X$.
	The same procedure, with~$z$ in place of~$z^\perp$, shows that also the value of~$B_j$ does not depend on~$X$.
	
	The polynomial~$\pol\big(g_0(\genvec)\big)$ has~$x_j=(\genvec,\basevec_j)$ as variables, hence~$\pol\big(g_0\circ g(\genvec)\big)$ is a polynomial of variables~$\big(\genvec,g^{-1}(\basevec_j)\big)$, for every~$g\in G$.
	To construct the polynomials~$\pol_{\borw,h^+,h^-}$, we need to split~$g^{-1}(\basevec_j)$ as in~\eqref{eq;splitg-1v_j}, replace these in the variables of~$\pol\big(g_0\circ g(\genvec)\big)$, and gather all factors of the form~$(\genvec,\genU_{z^\perp})$ and~$(\genvec,\genU_z)$.
	We then deduce that~${\pol_{\borw,h^+,h^-}\big(g_0\circ\borw(\genvec)\big)}$ is a function of~$A_j(g)$, $B_j(g)$ and~$\big(\genvec,g^{-1}(\basevec_j)_{w^\perp\oplus w}\big)$, where~$j$ runs from~$1$ to~$b+2$.
	
	We want to prove that~$\pol_{\borw,h^+,h^-}\big(g_0\circ\borw(\lambda)\big)$ does not depend on the real part~$X$, for every~${\lambda\in\brK\otimes\RR}$, where we identify~$g=\iden(\kappa,Z)$.
	We already proved that~$A_j$ and~$B_j$ does not depend on~$X$.
	We rewrite
	\be\label{eq;lamg-1v_j}
	\Big(\lambda,g^{-1}(\basevec_j)_{w^\perp\oplus w}\Big) = \Big(\lambda_{w^\perp\oplus w},g^{-1}(\basevec_j)\Big) = \Big( g(\lambda_{w^\perp\oplus w}),\basevec_j\Big)=\Big(\kappa\cdot\borwiden{1}{Z}(\lambda),\basevec_j\Big),
	\ee
	and remark that the right-hand side of~\eqref{eq;lamg-1v_j} does not depend on~$X$ by Lemma~\ref{lemma;choosegoodiden}.
	This concludes the proof.
	\end{proof}
	\begin{proof}[Proof of Theorem~\ref{thm;Fourexpofcoef}]
	We consider the unfolding of~$\intfunct$ provided by Corollary~\ref{cor;unfolded!}.
	The first summand of the right-hand side of~\eqref{eq;firstunftrickap} is part of the constant term of the Fourier expansion of~$\intfunct$, since it does not depend on~$X$.
	In fact, by Lemma~\ref{lemma;someformufortubdom}, we may rewrite it with respect to the identification~$\iden$ as
	\ba\label{eq;proofofFexpgen1}
	\int_{\SL_2(\ZZ)\backslash\HH}&\frac{y^{k+1/2}f(\tau)}{\sqrt{2}\,|\genU_{z^\perp}|}
	\,
	\overline{\Theta_{\brK}(\tau,\borw,\polw{(\alpha,\beta),\borw}{0}{0})}\frac{dx\,dy}{y^2}
	\\
	=&\int_{\SL_2(\ZZ)\backslash\HH}\frac{y^{k+1/2}f(\tau)|Y|}{\sqrt{2}}\sum_{\lambda\in\brK}\exp(-\Delta/8\pi y)\big(\polw{(\alpha,\beta),\borw}{0}{0}\big)\big(g_0\circ\borw(\lambda)\big)
	\\
	&\times e\big( -xq(\lambda)\big)\cdot \exp\big(-\pi y\lambda^2+2\pi y(\lambda,Y)^2/Y^2\big)\frac{dx\,dy}{y^2}.
	\ea
	Lemma~\ref{lemma;nodependenceonXgen1} implies that such a value does not depend on~$X$.
	
	As we are going to show soon, all other non-zero Fourier coefficients of the remaining summand~$\int_{\Gamma_\infty\backslash\HH}\hhab(\tau,g)\frac{dx\,dy}{y^2}$ of~\eqref{eq;firstunftrickap} correspond to some $\lambda\in\brK$ of positive norm, so that~$e\big((\lambda,X)\big)$ is not a constant function.
	This implies that~\eqref{eq;proofofFexpgen1} is exactly the constant term of the Fourier expansion of~$\intfunct$.
	
	We now begin the computation of the Fourier expansion of the second summand appearing on the right-hand side of~\eqref{eq;firstunftrickap}.
	First of all, we compute the series expansion with respect to~$\tau=x+iy\in\HH$ of~$f(\tau)\cdot\overline{\Theta_{\brK}(\tau,r\mu,0,\borw,\polw{(\alpha,\beta),\borw}{h^+}{0})}$.
	To do so, we replace~$f$ and~$\Theta_{\brK}$ respectively with~\eqref{eq;fexpf} and~\eqref{eq;precFexpofTllor}, deducing that such a product equals
	\begin{align*}
	\sum_{m\in\ZZ}
	&
	\sum_{\substack{n>0,\,\lambda\in\brK\\ n-q(\lambda)=m}}c_n(f)\cdot\exp(-2\pi ny)\cdot\exp(-\Delta/8\pi y)\big(\polw{(\alpha,\beta),\borw}{h^+}{0}\big)\big(g_0\circ\borw(\lambda)\big)
	\\
	&\times \exp\big(-\pi y(\lambda,\lambda)_w\big)\cdot e\big(r(\lambda,\mu)\big)
	\cdot e(mx).
	\end{align*}
	
	We insert the previous formula in the defining formula of~$\hhab$ provided by Proposition~\ref{prop;compauxhab}, and then replace this in the second summand of the right-hand side of~\eqref{eq;firstunftrickap} deducing that
	\ba\label{eq;killingx}
	2\int_{\Gamma_\infty\backslash\HH}
	&
	\hhab(\tau,g)\frac{dx\,dy}{y^2}=
	\frac{\sqrt{2}}{|\genU_{z^\perp}|}
	\sum_{h^+=0}^2
	\sum_{r\ge1}
	\Big(\frac{r}{2i}\Big)^{h^+}
	\sum_{m\in\ZZ}\sum_{\substack{n>0, \lambda\in\brK\\ n-q(\lambda)=m}}
	c_n(f)
	\\
	&\times  e\big(r(\lambda,\mu)\big) \int_{0}^{+\infty}y^{k-h^+-3/2}\exp\Big(-2\pi ny - \pi y(\lambda,\lambda)_w	- \frac{\pi r^2}{2y\genU_{z^\perp}^2}
	\Big)
	\\
	&\times \exp(-\Delta/8\pi y)\big(\polw{(\alpha,\beta),\borw}{h^+}{0}\big)\big(g_0\circ\borw(\lambda)\big) dy  \int_0^1 
	e(mx)dx.
	\ea
	Since~$\int_0^1 e(mx)dx$ equals~$1$ if~$m=0$ and is trivial otherwise, we may simplify~\eqref{eq;killingx} extracting the terms associated with~${m=0}$, obtaining that
	\ba\label{eq;killedx}
	2
	&
	\int_{\Gamma_\infty\backslash\HH}\hhab(\tau,g)\frac{dx\,dy}{y^2}=
	\frac{\sqrt{2}}{|\genU_{z^\perp}|}
	\sum_{\lambda\in\brK} c_{q(\lambda)}(f)
	\sum_{h^+=0}^2
	\sum_{r\ge1}
	\Big(\frac{r}{2i}\Big)^{h^+}	
	\int_{0}^{+\infty}y^{k-h^+-3/2}
	\\
	&\times\exp\Big(-2\pi y \lambda_{w^\perp}^2 -
	\frac{\pi r^2}{2y\genU_{z^\perp}^2}\Big)\cdot\exp(-\Delta/8\pi y)\big(\polw{(\alpha,\beta),\borw}{h^+}{0}\big)\big(g_0\circ\borw(\lambda)\big)dy \cdot e\big(r(\lambda,\mu)\big).
	\ea
	
	Since~$e\big((\lambda,\mu)\big)=e\big((\lambda,X)\big)$ by Lemma~\ref{lemma;someformufortubdom}, we may rewrite~\eqref{eq;killedx} in the same shape of~\eqref{eq;foucoefcor0}, i.e.\ we gather the terms multiplying~$e\big((\lambda,\mu)\big)$, for every~$\lambda$.
	This can be done simply replacing the sum~$\sum_{r\ge1}$ with~$\sum_{t\ge1,\,t|\lambda}$, and the lattice vector~$\lambda$ with~$\lambda/t$.
	In this way, we obtain that
	\ba\label{eq;FseriesexpofFgen1}
	2
	&
	\int_{\Gamma_\infty\backslash\HH}\hhab(\tau,g)\frac{dx\,dy}{y^2}=
	\frac{\sqrt{2}}{|\genU_{z^\perp}|}
	\sum_{\lambda\in\brK}
	\sum_{h^+=0}^2
	\sum_{\substack{t\ge 1\\ t|\lambda}}
	\Big(\frac{t}{2i}\Big)^{h^+}
	c_{q(\lambda/t)}(f)
	\int_0^{+\infty} y^{k-h^+-3/2}
	\\
	&\times \exp\Big(-\frac{2\pi y \lambda_{w^\perp}^2}{t^2} - 
	\frac{\pi t^2}{2y\genU_{z^\perp}^2}
	\Big)\cdot \exp(-\Delta/8\pi y)\big(\polw{(\alpha,\beta),\borw}{h^+}{0}\big)\big(g_0\circ\borw(\lambda/t)\big)dy\cdot e\big((\lambda,\mu)\big).
	\ea
	This is the Fourier expansion of~$2\int_{\Gamma_\infty\backslash\HH}\hhab(\tau,g)\frac{dx\,dy}{y^2}$.
	In fact, if we identify~$G$ with~$K\times\tubedom$ under~$\iden$, and write~$g=\iden(\kappa,Z)$, then by Lemma~\ref{lemma;someformufortubdom} we may rewrite~\eqref{eq;FseriesexpofFgen1} as
	\begin{align}\label{eq;killedxintubedomain}
	& 2 \int_{\Gamma_\infty\backslash\HH}\hhab(\tau,g)\frac{dx\,dy}{y^2}\\
	&
	=\sqrt{2}\,|Y|\sum_{\lambda\in\brK}
	\sum_{\text{$t\ge 1$, $t|\lambda$}}
	c_{q(\lambda/t)}(f)
	\sum_{h^+=0}^2
	\Big(\frac{t}{2i}\Big)^{h^+}
	\int_{0}^{+\infty}y^{k-h^+-3/2}
	\exp\Big(-\frac{2\pi y \lambda^2}{t^2}\Big)
	\nonumber
	\\
	& \quad\times \exp\Big(\frac{2\pi y(\lambda,Y)^2}{t^2 Y^2} + \frac{\pi t^2 Y^2}{2y}\Big) \exp(-\Delta/8\pi y)\big(\polw{(\alpha,\beta),\borw}{h^+}{0}\big)\big(g_0\circ\borw(\lambda/t)\big) dy
	\cdot e\big((\lambda,X)\big)\nonumber.
	\end{align}
	By Lemma~\ref{lemma;nodependenceonXgen1} the coefficient multiplying~$e\big((\lambda,X)\big)$ in~\eqref{eq;killedxintubedomain} does not depend on~$X$.
	\end{proof}
	
	\section{The injectivity of the Kudla--Millson lift}\label{sec;injKMgenus1}
	
	This section is devoted to the proof of the injectivity of the Kudla--Millson lift~$\KMliftbase$ associated to \emph{unimodular} lattices of signature~$(b,2)$.
	Although such a result has already been proved in~\cite{br;converse}, the procedure here proposed has the advantage of paving the ground for a strategy to prove the injectivity of the lift in several other cases; see~\cite{metzlerzuffetti} and~\cite{kieferzuffetti}.
	The case of non-unimodular lattices is carried out in Section~\ref{sec;nonunimodlattgen1}.
	\begin{thm}\label{thm;injindeg1}
	Let $L$ be a unimodular lattice of signature $(b,2)$, with $b>2$.
	The Kudla--Millson theta lift $\KMliftbase$ associated to $L$ is injective. 
	\end{thm}
	To prove Theorem~\ref{thm;injindeg1}, we need the following ancillary result.
	\begin{lemma}\label{lemma;technnonzeropolinprinj}
	Let $\lambda\in\brK\otimes\RR$ be such that $q(\lambda)>0$.
	There exist two different indexes~${\alpha,\beta\in\{1,\dots,b-1\}}$, and $g\in G$, such that
	\bes
	\polw{(\alpha,\beta),\borw}{1}{0}\big(g_0\circ\borw(\lambda)\big) > 0.
	\ees
	\end{lemma}
	\begin{proof}[Proof of Lemma~\ref{lemma;technnonzeropolinprinj}]
	Recall from Section~\ref{sec;follBorcherds} that we may use the standard basis vectors~$\basevec_j$ of~$L\otimes\RR$ to construct a basis of the subspace $\brK\otimes\RR$ as $\basevec_1,\dots,\basevec_{b-1},\basevec_{b+1}$.
	We rewrite the vector~$\lambda\in\brK\otimes\RR$ with respect to such a basis as
	\bes
	\lambda=\sum_{j=1}^{b-1}\lambda_j \basevec_j+\lambda_{b+1}\basevec_{b+1},
	\ees
	for some real coefficients $\lambda_j$.
	Since
	\bes
	2q(\lambda)=\sum_{j=1}^{b-1}\lambda_j^2-\lambda_{b+1}^2,
	\ees
	and since $q(\lambda)>0$ by assumption, there exists an index $\beta\in\{1,\dots,b-1\}$ such that the~$\beta$-th coordinate $\lambda_\beta$ of $\lambda$ is \textcolor{\myblack}{non-zero}.
	
	Let $\alpha\in\{1,\dots,b-1\}$ be such that $\alpha\neq\beta$.
	Recall from~\eqref{eq;comppolPab} that~$\polab\big(g_0(\genvec)\big)=2x_\alpha x_\beta,$ for every~$\genvec=\sum_j x_j\basevec_j\in L\otimes\RR$.
	We define~$g\in G$ to be the isometry interchanging~$\basevec_\alpha$ with~$\basevec_b$, and~$\basevec_{b+1}$ with~$\basevec_{b+2}$, fixing the remaining standard basis vectors.
	We remark that~$g$ is an element of the stabilizer~$K$ of the base point~${z_0\in \Gr(L)}$.
	For this choice of~$g$ we deduce that~$\polab(g_0\circ g(\genvec))=2x_b x_\beta$, since
	\bes
	g(\genvec)=\sum_{j=1}^{b+2}x_j g(\basevec_j)=x_1\basevec_1+\dots+x_b \basevec_\alpha+\dots+x_\alpha \basevec_b+\dots+x_{b+2}\basevec_{b+1}+x_{b+1}\basevec_{b+2}.
	\ees
	
	We write~$\polab$ as in~\eqref{eq;decpolborh-0}, for some homogeneous polynomials~${\polw{(\alpha,\beta),\borw}{h^+}{0}}$ of degree respectively $(2-h^+,0)$ on the subspaces~$g_0\circ \borw(L\otimes\RR)\cong\RR^{b-1,1}$.
	Since we chose~${u=(\basevec_b + \basevec_{b+2})/\sqrt{2}}$, and since the base point~$z_0$ of~$\Gr(L)$, stabilized by~$g$, is the negative definite plane in~$L\otimes\RR$ generated by~$\basevec_{b+1}$ and~$\basevec_{b+2}$, we deduce that~$u_{z_0^\perp}=\basevec_b/\sqrt{2}$.
	This implies that~$(\genvec,u_{z_0^\perp})=x_b/\sqrt{2}$, hence we deduce that
	\be\label{eq;prooftecxbeta}
	\polab\big(g_0\circ g(\genvec)\big)=(\genvec,u_{z_0^\perp})\cdot 2\sqrt{2}x_\beta.
	\ee
	
	If we compare~\eqref{eq;prooftecxbeta} with~\eqref{eq;decpolborh-0}, or alternatively use Lemma~\ref{lemma;rewritborchimpldecourpol}, we see that for this special choice of~$g$ we have
	\bes
	\polw{(\alpha,\beta),\borw}{h^+}{0}\big(g_0\circ\borw(\genvec)\big)=\begin{cases}
	2\sqrt{2}x_\beta, & \text{if $h^+=1$,}\\
	0, & \text{otherwise.}
	\end{cases}
	\ees
	Since we chose $\beta$ such that the $\beta$-th coordinate of $\lambda$ is positive, we than conclude that~$\polw{(\alpha,\beta),\borw}{1}{0}\big(g_0\circ\borw(\lambda)\big)>0$.
	\end{proof}
	
	We are now ready to prove the main result of this section.
	
	\begin{proof}[Proof of Theorem~\ref{thm;injindeg1}]
	Let $f\in S^k_1$ be such that $\KMliftbase(f)=0$.
	We want to prove that this implies $f=0$.	
	Recall from~\eqref{eq;KMlambdaaO} that
	\be\label{eq;injrec}
	\KMliftbase(f)=\sum_{\alpha,\beta=1}^{b}\intfunct(g) \cdot g^*\Big(\omega_{\alpha, b+1}\wedge\omega_{\beta, b+2}\Big),\quad\text{for every~$g\in G$.}
	\ee
	Since the vectors~$\omega_{\alpha, b+1}\wedge\omega_{\beta, b+2}$, where $\alpha,\beta=1,\dots,b$, are linearly independent in~${\bigwedge}^2(\mathfrak{p})^*$, we deduce from~\eqref{eq;injrec} that~$\KMliftbase(f)=0$ if and only if the defining integrals~$\intfunct$ are all zero, namely
	\be\label{eq;injreccoeffab}
	\int_{\SL_2(\ZZ)\backslash\HH}y^{k+1}f(\tau)\,\overline{\Theta_L(\tau,g,\polab)}\,\frac{dx\,dy}{y^2}=0,\qquad\text{for every $\alpha,\beta$ and for every $g\in G$.}
	\ee

	As complex valued functions on~$G$, the defining integrals~$\intfunct\colon G\to\CC$ of the Kudla--Millson lift of~$f$ admit a Fourier expansion in the sense of Section~\ref{sec;FexpLlorinvfuncts}.
	By Theorem~\ref{thm;Fourexpofcoef}, the Fourier expansion of such defining integrals is
	\ba\label{eq;Fexpfininproofinj}
	\intfunct(g)
	& =
	\int_{\SL_2(\ZZ)\backslash\HH}
	\frac{y^{k+1/2}f(\tau)}{\sqrt{2}\,|\genU_{z^\perp}|}
	\overline{\Theta_{\brK}(\tau,\borw,\polw{(\alpha,\beta),\borw}{0}{0})}\frac{dx\,dy}{y^2}
	\\
	&\quad +
	\frac{\sqrt{2}}{|\genU_{z^\perp}|}\sum_{\lambda\in\brK}\sum_{h^+=0}^2
	\sum_{\substack{t\ge 1\\ t|\lambda}}
	\Big(\frac{t}{2i}\Big)^{h^+} c_{q(\lambda)/t^2}(f)
	\int_0^{+\infty}
	\exp\Big(-\frac{2\pi y \lambda_{w^\perp}^2}{t^2} -
	\frac{\pi t^2}{2y\genU_{z^\perp}^2}
	\Big)
	\\
	&\quad\times 
	y^{k-h^+-3/2}
	\exp(-\Delta/8\pi y)\big(\polw{(\alpha,\beta),\borw}{h^+}{0}\big)\big(g_0\circ\borw(\lambda/t)\big)dy\cdot e\big((\lambda,\mu)\big).
	\ea
	We deduce from~\eqref{eq;injreccoeffab} that the Fourier coefficients of the Fourier expansion~\eqref{eq;Fexpfininproofinj} are all zero.	
	We want to use this to show that $c_n(f)=0$ for every positive integer $n$, that is, the cusp form $f$ is zero.
	
	We work by induction on the divisibility of all lattice vectors~$\lambda\in\brK$ such that~$q(\lambda)>0$.
	Suppose that~$\lambda$ is \emph{primitive}, that is, the only integer~$t\ge1$ dividing~$\lambda$ is~$t=1$.
	The fact that the Fourier coefficient of~\eqref{eq;Fexpfininproofinj} associated to~$\lambda$ equals zero means that
	\ba\label{eq;injreccorFcoefffc}
	&\frac{\sqrt{2} c_{q(\lambda)}(f)}{|\genU_{z^\perp}|}
	\sum_{h^+=0}^2 (2i)^{-h^+}
	\int_0^{+\infty}
	y^{k-h^+-3/2}
	\exp\Big(-2\pi y \lambda_{w^\perp}^2 -
	\frac{\pi}{2yu_{z^\perp}^2}
	\Big)
	\\
	&\quad\times
	\exp(-\Delta/8\pi y)\big(\polw{(\alpha,\beta),\borw}{h^+}{0}\big)\big(g_0\circ\borw(\lambda)\big)
	dy=0.
	\ea
	Note that the integral appearing in~\eqref{eq;injreccorFcoefffc} is a \emph{real} number.
	
	We are going to prove that there exist two different indices $\alpha,\beta\in\{1,\dots,b-1\}$ and an isometry $g\in G$, such that the sum over $h^+$ appearing on the left-hand side of~\eqref{eq;injreccorFcoefffc} is non-zero.
	This implies that~$c_{q(\lambda)}(f)=0$, concluding the first step of the induction.
	
	By Lemma~\ref{lemma;technnonzeropolinprinj} there exist two different indices~$\alpha,\beta$, and an isometry~$g\in G$, such that~$\polw{(\alpha,\beta),\borw}{1}{0}\big(g_0\circ\borw(\lambda)\big)\neq 0$.
	This implies that, for such choices of~$\alpha,\beta$ and~$g$, the sum over~$h^+$ appearing on the left-hand side of~\eqref{eq;injreccorFcoefffc} is a non-zero complex number.
	In fact, its imaginary part is
	\ba\label{eq;realpartproofinj}
	-\frac{1}{2}\polw{(\alpha,\beta),\borw}{1}{0}\big(g_0\circ\borw(\lambda)\big)\cdot\int_0^{+\infty} y^{k-5/2}\cdot \exp\Big(-2\pi y \lambda_{w^\perp}^2 -
	\frac{\pi}{2yu_{z^\perp}^2}
	\Big)dy.
	\ea
	Note that the integral appearing in~\eqref{eq;realpartproofinj} is a \emph{positive} real number.
	We remark that in~\eqref{eq;realpartproofinj} we drop the operator~$\exp(-\Delta/8\pi y)$ acting on~$\polw{(\alpha,\beta),\borw}{1}{0}$, since the latter is a polynomial of degree one, hence harmonic.
	
	We now use induction.
	Suppose that $c_{q(\lambda')}(f)=0$ for every $\lambda'\in\brK$ divisible by at most~$s$ positive integers.
	Let $\lambda\in\brK$ be such that it is divisible by~$s+1$ integers~${1<d_1<\dots <d_s}$.
	Since $c_{q(\lambda/d_j)}(f)=0$ for every $j=1,\dots,s$ by inductive hypothesis, we may simplify the formula of the Fourier coefficient associated to $\lambda$ of the Fourier expansion~\eqref{eq;Fexpfininproofinj} again as~\eqref{eq;injreccorFcoefffc}, where this time $\lambda$ is non-primitive.
	Since the primitivity of~$\lambda$ does not play any role in Lemma~\ref{lemma;technnonzeropolinprinj}, we may deduce~$c_{q(\lambda)}(f)=0$ with the same procedure used for the case of primitive~$\lambda$.
	
	To conclude the proof, it is enough to show that for every positive integer~$n$ there exists~$\lambda\in\brK$ such that~$n=q(\lambda)$, and hence~$c_n(f)=0$ by the previous inductive argument.
	Equivalently, we want to prove that the quadratic form of the lattice~$\brK$ represents every positive integer.
	This is ensured from the unimodularity of~$\brK$, since this implies that~$\brK$ splits off an hyperbolic plane.
	In fact, it is well-known that the quadratic form of an hyperbolic plane represents all positive integers.\hfill\qedhere
	\end{proof}
		
	\section{The case of non-unimodular lattices}\label{sec;nonunimodlattgen1}
	
We illustrated above how to prove the injectivity of the Kudla--Millson lift in the case of even unimodular lattices of signature~$(b,2)$.
	In this section we describe what needs to be changed to deal with non-unimodular lattices.
	In particular, we provide a proof of the injectivity of the Kudla--Millson lift~$\KMliftbase$ in the case of (not necessarily unimodular) lattices~$L$ that split off~$U(N)\oplus U$, for some positive integer~$N$.
	This result is as in~\cite[Theorem~$5.3$]{br;converse}, but proved here in a different way.
	
	The procedure that we follow in this section is essentially the same as the one used in the previous sections.
	This motivates why we emphasize here only the main differences with respect to the previous easier case, without providing the same amount of details.\\
	
	Throughout this section we denote by~$L$ a (not necessarily unimodular) even lattice of signature~$(b,2)$, where~$b>2$, and we set~$k=1+b/2\in\frac{1}{2}\ZZ$.
	The \emph{discriminant group} associated to~$L$ is the quotient~$L'/L$, where~$L'$ is the dual of~$L$.
	The quadratic form~$q$ of~$L$ induces a~$\QQ/\ZZ$-valued quadratic form on~$L'/L$, which we still denote by~$q$.
	
	We denote by~$(\mathfrak{e}_h)_{h\in L'/L}$ the standard basis of the group algebra~$\CC[L'/L]$, and by~$\langle\cdot{,}\cdot\rangle$ the standard scalar product of~$\CC[L'/L]$ defined as
	\bes
	\Big\langle
	\sum_{h\in L'/L}\lambda_h\mathfrak{e}_h,\sum_{h\in L'/L}\mu_h\mathfrak{e}_h
	\Big\rangle \coloneqq \sum_{h\in L'/L}\lambda_h\overline{\mu_h}.
	\ees
	
	Let~$\rho_L$ be the Weil representation of the metaplectic group~$\Mp_2(\ZZ)$ on~$\CC[L'/L]$; see~\cite[Section~$1.1$]{br;borchp} for details.
	A \emph{weight~$k$ modular form with respect to~$\rho_L$ and~$\Mp_2(\ZZ)$} is a function~$f\colon\HH\to\CC[L'/L]$ which is holomorphic on~$\HH$ and at the cusp~$\infty$, and that satisfies the modularity law
	\bes
	f(\gamma\cdot\tau)=\phi(\tau)^{2k}\cdot\rho_L(\gamma,\phi)\cdot f(\tau),
	\ees
	for every~$(\gamma,\phi)\in\Mp_2(\ZZ)$ and every~$\tau\in\HH$.
	We denote the components of~$f$ by~$f_h$, so that~$f=\sum_{h\in L'/L}f_h\mathfrak{e}_h$.
	These vector-valued modular forms admit a Fourier expansion, which we write as
	\bes
	f(\tau)=
	\sum_{h\in L'/L}\sum_{\substack{n\in\ZZ+q(h)\\ n\ge 0}}c_n(f_h)\exp(-2\pi ny)e(nx)\mathfrak{e}_h,
	\ees
	where $c_n(f_h)$ is the $n$-th Fourier coefficient of $f_h$, or equivalently the $n$-th Fourier coefficient of index~$h$ of~$f$.
	If all~$c_0(f_h)$ vanish, then~$f$ is called a \emph{cusp form}.
	We denote by~$M^k_{1,L}$, resp.~$S^k_{1,L}$, the space of modular forms, resp.\ cusp forms, of weight~$k$ with respect to~$\rho_L$ and~$\Mp_2(\ZZ)$.
	
	We may rewrite the Kudla--Millson theta form attached to the lattice~$L$ as
	\ba\label{eq;KMthetaformvv}
	&\Theta(\tau, z,\varphi_{\text{\rm KM}})
	\\
	&
	=y^{-k/2}\sum_{h\in L'/L}\sum_{\lambda\in L+h}
	\Big(\omega_\infty(g_\tau)\varphi_{\text{\rm KM}}\Big)(\lambda,z)
	\mathfrak{e}_h
	\\
	&= \sum_{\alpha,\beta=1}^b\underbrace{y^{-k/2}\sum_{h\in L'/L}\sum_{\lambda\in L+h}\Big(\omega_\infty(g_\tau)(\Gpol_{(\alpha,\beta)}\varphi_0)\Big)\big(g_0\circ g(\lambda)\big)\mathfrak{e}_h}_{\eqqcolon \FFab(\tau,g)}
	\textcolor{\myblack}{\otimes}\,
	g^*(\omega_{\alpha,b+1}\wedge\omega_{\beta,b+2}),
	\ea
	where~$g\in G$ is any isometry mapping~$z\in\hermdom=\Gr(L)$ to the base point~$z_0$, and~$\Gpol_{(\alpha,\beta)}$ is the polynomial on~$\RR^{b,2}$ defined in~\eqref{eq;comppolPab}.
	The auxiliary function~$\FFab$ highlighted in~\eqref{eq;KMthetaformvv} may be rewritten in terms of \emph{vector-valued} Siegel theta functions as
	\bas
	\FFab(\tau,g)
	&=
	y\cdot\sum_{h\in L'/L}\sum_{\lambda\in L+h}\exp(-\Delta/8\pi y)(\polab)\big(g_0\circ g(\lambda)\big)\cdot e\Big(
	\tau q(\lambda_{z^\perp})+\bar{\tau}q(\lambda_z)
	\Big)\mathfrak{e}_h
	\\
	&=y\cdot\Theta_L(\tau,g,\polab).
	\eas
	We suggest the reader to recall the vector valued theta functions~$\Theta_L$, together with their modular transformation properties, from~\cite[Section~$4$]{bo;grass}.
	Whenever~$L$ is unimodular, they are exactly the ones introduced in Section~\ref{sec;Siegtheta}.
	
	The Kudla--Millson lift~$\KMliftbase\colon S^k_{1,L}\to\mathcal{Z}^2(\hermdom)$ is defined as
	\be\label{eq;liftingLambdavv}
	f\longmapsto\KMliftbase(f)=
	\int_{\SL_2(\ZZ)\backslash\HH}
	y^k \langle f(\tau),\Theta(\tau,z,\varphi_{\text{\rm KM}})\rangle\frac{dx\, dy}{y^2},
	\ee
	where $\frac{dx\, dy}{y^2}$ is the standard~$\SL_2(\ZZ)$-invariant volume element of~$\HH$.
	We may rewrite such a lift by means of~\eqref{eq;KMthetaformvv} as
	\ba\label{eq;KMlambdaaOvv}
	\KMliftbase(f)=\sum_{\alpha,\beta=1}^{b}\Big(
	\underbrace{\int_{\SL_2(\ZZ)\backslash\HH} y^{k+1} \langle f(\tau),\Theta_L(\tau,g,\polab)\rangle\frac{dx\,dy}{y^2}}_{\eqqcolon\intfunct(g)}
	\Big)
	\cdot
	g^*\Big(\omega_{\alpha, b+1}\wedge\omega_{\beta, b+2}\Big).
	\ea
	
	We refer to the integrals~$\intfunct$ appearing in~\eqref{eq;KMlambdaaOvv} as the \emph{defining integrals} of the lift~$\KMliftbase(f)$.
	We want to compute their Fourier expansions by applying the unfolding trick.
	To do this, we need to introduce another piece of notation, following the wording of~\cite[pp.~$41$-$42$]{br;borchp}.
	Recall that we do \emph{not} assume that~$L$ splits off any hyperbolic plane, for now.
	
	Let~$\genU$ be a primitive norm~$0$ vector of~$L$, and let~$\genUU\in L'$ be such that~${(\genU,\genUU)=1}$.
	Define~$\brK=(L\cap\genU^\perp)/\ZZ \genU$, and write~$\brN$ for the smallest positive value of the inner product of~$\genU$ with something in~$L$, so that~$|L'/L|=\brN^2|\brK'/\brK|$.
	Let~$L_0'$ be the sublattice of~$L'$ defined as
	\bes
	L_0'=\{\text{$\lambda\in L'$ : $(\lambda,\genU)\equiv 0$ mod $\brN$}\}.
	\ees
	We consider the projection~$p\colon L_0'\to\brK'$ constructed in~\cite[$(2.7)$]{br;borchp}.
	This map is such that~$p(L)=\brK$, and induces a surjective map~$L_0'/L\to\brK'/\brK$ which we also denote by~$p$.
	We recall that~$L_0'/L=\{\text{$\lambda\in L'/L$ : $(\lambda,\genU)\equiv 0$ mod $\brN$}\}$.
	
	By~\cite[Theorem~$5.2$]{bo;grass} we may rewrite the integrand of~$\intfunct$ as
	\bas
	y^{k+1} \langle f(\tau),&\Theta_L(\tau,g,\polab)\rangle
	=
	\frac{y^{k+1/2}}{\sqrt{2}\,|\genU_{z^\perp}|}
	\big\langle
	f_{\brK}(\tau;0,0),\Theta_{\brK}(\tau,\borw,\polw{(\alpha,\beta),\borw}{0}{0})
	\big\rangle
	\\
	&
	+ \frac{y^{k+1/2}}{\sqrt{2}\,|\genU_{z^\perp}|}\sum_{\substack{c,d\in\ZZ\\ \gcd(c,d)=1}}
	\sum_{r\ge1}
	\sum_{h^+=0}
	\Big(\frac{r}{2iy}\Big)^{h^+}
	(c\tau+d)^{h^+} e\Big(
	-\frac{r^2 |c\tau+d|^2}{4iy\genU_{z^\perp}^2}
	\Big)
	\\
	&\times \big\langle
	f_{\brK}(\tau;-rd,rc),\Theta_{\brK}(\tau,rd\mu,-rc\mu,\borw,\polw{(\alpha,\beta),\borw}{h^+}{0})
	\big\rangle,
	\eas
	where~$f_{\brK}(\tau;r,t)$ is the function arising from~$f\in S^k_{1,L}$ constructed as in~\cite[($2.12$)]{br;borchp}.
	
	Let~$\hhab$ be the auxiliary~$\Gamma_\infty$-invariant function defined as
	\bas
	\hhab(\tau,g)
	&=
	\frac{y^{k+1/2}}{\sqrt{2}\,|\genU_{z^\perp}|}\sum_{r\ge1}\sum_{h^+=0}^2 (2iy)^{-h^+} r^{h^+} \exp\Big(
	-\frac{\pi r^2}{2y\genU_{z^\perp}^2}
	\Big)
	\\
	&\quad\times
	\big\langle f_{\brK}(\tau;-r,0),\Theta_{\brK}(\tau,r\mu,0,\borw,\polw{(\alpha,\beta),\borw}{h^+}{0})\big\rangle,
	\eas
	for every~$\tau\in\HH$ and~$g\in G$, where~$z=g^{-1}(z_0)\in\Gr(L)$.
	Following the same procedure of Proposition~\ref{prop;compauxhab}, together with~\cite[Theorem~$2.6$]{br;borchp}, we may deduce that
	\bas
	\langle f(\tau),\FFab(\tau,g)\rangle y^k
	&=
	\frac{y^{k+1/2}}{\sqrt{2}\,|\genU_{z^\perp}|}\langle f_{\brK}(\tau,0,0),\Theta_{\brK}(\tau,\borw,\polw{(\alpha,\beta),\borw}{0}{0})\rangle
	\\
	&\quad
	+
	\sum_{\gamma=\left(\begin{smallmatrix}
	* & *\\ c & d
	\end{smallmatrix}\right)\in\Gamma_\infty\backslash\SL_2(\ZZ)}\hhab(\gamma\cdot\tau,g).
	\eas
	
	We proceed with the unfolding of~$\intfunct$.
	We may deduce that the Fourier coefficient of~$\intfunct$ associated to~$\lambda\in\brK+h_{\Lor}$, where~$h_{\Lor}\in\brK'/\brK$ and~$q(\lambda)>0$, equals
	\ba\label{eq;Fcoeffq(brlam)>0vv}
	&\frac{\sqrt{2}}{|\genU_{z^\perp}|}
	\sum_{h^+=0}^2
	\sum_{\substack{t\in\ZZ_{>0}\\ t|\lambda}}
	\Big(
	\frac{t}{2i}
	\Big)^{h^+}
	\sum_{\substack{\nuni\in L_0'/L \\ p(h)=\nunilor/t}}e\big( t(\nuni,\genUU)\big)\cdot c_{q(\lambda)/t^2}(f_{\nuni})
	\int_0^{+\infty} y^{k-h^+-3/2}
	\\
	&\quad\times 
	\exp\Big(-\frac{2\pi y \lambda_{w^\perp}^2}{t^2} -
	\frac{\pi t^2}{2y\genU_{z^\perp}^2}
	\Big) \cdot
	\exp(-\Delta/8\pi y)\big(\polw{(\alpha,\beta),\borw}{h^+}{0}\big)\big(g_0\circ\borw(\lambda/t)\big)dy,
	\ea
	where \textcolor{\myblack}{we say that a positive integer $t$ divides $\lambda\in\brK'$, in short~$t|\lambda$, if and only if~$\lambda/t\in\brK'$}.
	\begin{thm}(Bruinier)\label{thm;injindeg1vv}
	Let $L$ be an even lattice of signature $(b,2)$, with $b>2$, that splits off~${U(N)\oplus U}$, for some positive integer~$N$.
	The Kudla--Millson theta lift $\KMliftbase$ associated to $L$ is injective.
	\end{thm}
	Since a large part of the proof of Theorem~\ref{thm;injindeg1vv} is essentially the same as the one of Theorem~\ref{thm;injindeg1}, we provide only a sketch of it.
	\begin{proof}[Sketch of the proof]
	Let~$f\in S^k_{1,L}$ be such that~$\KMliftbase(f)=0$.
	This is equivalent to saying that~$\intfunct(g)=0$ for every~$\alpha,\beta$ and every~$g\in G$, since the vectors~$\omega_{\alpha,b+1}\wedge\omega_{\beta,b+2}$ appearing in~\eqref{eq;KMlambdaaOvv} are linearly independent in~$\bigwedge^2(\mathfrak{p}^*)$.
	We want to show that the vanishing of the defining integrals of~$\KMliftbase(f)$ implies that~$f=0$.
	
	The Fourier coefficients of~$\intfunct$ associated to~$\lambda\in\brK+h_{\Lor}$, where~$h_{\Lor}\in\brK'/\brK$ and~$q(\lambda)>0$, are as in~\eqref{eq;Fcoeffq(brlam)>0vv}.
	These coefficients are all zero, since so is~$\intfunct$.
	We show that the vanishing of the Fourier coefficients of~$\intfunct$ implies the vanishing of the Fourier coefficients of~$f$ by induction on the divisibility of~$\lambda$.
	Suppose that~$\lambda$ is \emph{primitive}.
	The fact that the Fourier coefficient~\eqref{eq;Fcoeffq(brlam)>0vv} associated to~$\lambda$ equals zero is equivalent to
	\ba\label{eq;brlamprimvv}
	&\frac{\sqrt{2}}{|\genU_{z^\perp}|}\Big(
	\sum_{\substack{\nuni\in L_0'/L\\ p(\nuni)=\nunilor}} 
	e\big((\nuni,\genUU)\big)\cdot c_{q(\lambda)}(f_\nuni)
	\Big)\sum_{h^+=0}^2 (2i)^{-h^+}\int_0^{+\infty} y^{k-h^+-3/2}
	\\
	&\quad\times
	\exp\Big(-2\pi y \lambda_{w^\perp}^2 -
	\frac{\pi}{2yu_{z^\perp}^2}
	\Big)\cdot
	\exp(-\Delta/8\pi y)\big(\polw{(\alpha,\beta),\borw}{h^+}{0}\big)\big(g_0\circ\borw(\lambda)\big)
	dy
	=0.
	\ea
	
	Since~$L$ splits off a hyperbolic plane, we may choose~$\genU$ and~$\genUU$ to be the standard generators of such a hyperbolic plane, so that~$L=\brK\oplus U$ and~$U=\ZZ\genU\oplus\ZZ\genUU$.
	It is easy to see that~$L'/L\cong L_0'/L\cong \brK'/\brK$, that the map~$p$ is an isomorphism, and that the latter is actually the standard orthogonal projection~$L'/L\to\brK'/\brK$, $\nuni+L\to\nuni_{\brK} + \brK$.
	In particular, for every~$\nunilor\in\brK'/\brK$, the only~$\nuni\in L_0'/L$ such that~$p(\nuni)=\nunilor$ is~$\nuni=\nunilor+L$.
	
	Since~$\brK$ is orthogonal to~$\genUU$, an analogous argument on~\eqref{eq;brlamprimvv} as in the unimodular case shows that~$c_{q(\lambda)}(f_{\nunilor + L})=0$ for every primitive~$\lambda\in\brK+\nunilor$.
	This can be extended to every (not necessarily primitive)~$\lambda$ by an easy inductive argument.
	We then deduce that
	\be\label{eq;Cyan3vv}
	c_{q(\lambda)}(f_{\nunilor + L})=0,\qquad\text{for every~$\lambda\in\brK+\nunilor$ and~$\nunilor\in\brK'/\brK$.}
	\ee
	
	To conclude the proof, it is enough to show that~\eqref{eq;Cyan3vv} implies that
	\be\label{eq;Cyan4vv}
	c_{q(\lambda)}(f_\nuni)=0,\qquad\text{for every $\lambda\in L+\nuni$ and~$\nuni\in L'/L$.}
	\ee
	In fact~\eqref{eq;Cyan4vv} implies that~$c_n(f_\nuni)=0$ for every positive $n\in\ZZ+q(\nuni)$, since~$L$ splits off a hyperbolic plane.
	
	Our new approach provides a different way with respect to~\cite{br;converse} to prove~\eqref{eq;Cyan3vv}.
	In fact, the latter is the same as~\cite[(5.3)]{br;converse}.
	If~$L$ splits off two orthogonal hyperbolic planes, i.e.~$N=1$, then one can deduce~\eqref{eq;Cyan4vv} from~\eqref{eq;Cyan3vv} exactly as in~\cite[Proof of Theorem~$5.12$, last two paragraphs]{br;borchp}.
	In the more general case where~$L$ splits off~$U(N)\oplus U$, one can deduce~\eqref{eq;Cyan4vv} exactly as in~\cite[\textcolor{\myblack}{Proof of} Theorem~$5.3$, Part~$3$]{br;converse}.\qedhere

%
	\end{proof}

	
	\printbibliography

\end{document}